\documentclass[11pt,a4paper]{article}

\usepackage{epsf,epsfig,amsfonts,amsgen,amsmath,amstext,amsbsy,amsopn,amsthm,cases,listings,color
}
\usepackage{ebezier,eepic}
\usepackage{color}
\usepackage{multirow}
\usepackage{epstopdf}
\usepackage{graphicx}   
\usepackage{pgf,tikz}
\usetikzlibrary{patterns,patterns.meta} 
\usepackage{mathrsfs}
\usepackage[marginal]{footmisc}
\usepackage{enumerate}
\usepackage{enumitem}
\usepackage[titletoc]{appendix}
\usepackage{booktabs}
\usepackage{url}
\usepackage{mathtools}

\usepackage{pgfplots}
\usepackage{authblk}
\usepackage{amssymb}

\usepackage{wasysym}

\usepackage{empheq}

\usepackage{dsfont}

\usepackage{tikz}
\usepackage{longtable}
\usepackage{float}
\usepackage{subfigure}
\pgfplotsset{compat=1.18}
\usepackage{mathrsfs}
\usepackage{wasysym} 
\usetikzlibrary{arrows}
\usepackage{aligned-overset}
\usepackage{bm}
\usepackage{bbm}
\usepackage[T1]{fontenc}
\usepackage[backref=page]{hyperref}

\usepackage[thinc]{esdiff}

\allowdisplaybreaks[1]

\definecolor{uuuuuu}{rgb}{0.27,0.27,0.27}
\definecolor{sqsqsq}{rgb}{0.1255,0.1255,0.1255}

\newtheorem{definition}{Definition} [section]
\newtheorem{theorem}[definition]{Theorem}
\newtheorem{lemma}[definition]{Lemma}
\newtheorem{proposition}[definition]{Proposition}

\newtheorem{problem}[definition]{Problem}

\newtheorem{fact}[definition]{Fact}

\newcommand{\norm}[1]{\left\lVert#1\right\rVert}

\makeatletter
\newsavebox\myboxA
\newsavebox\myboxB
\newlength\mylenA
\newcommand*\xoverline[2][0.75]{%
    \sbox{\myboxA}{$\m@th#2$}%
    \setbox\myboxB\null
    \ht\myboxB=\ht\myboxA%
    \dp\myboxB=\dp\myboxA%
    \wd\myboxB=#1\wd\myboxA
    \sbox\myboxB{$\m@th\overline{\copy\myboxB}$}
    \setlength\mylenA{\the\wd\myboxA}
    \addtolength\mylenA{-\the\wd\myboxB}%
    \ifdim\wd\myboxB<\wd\myboxA%
       \rlap{\hskip 0.5\mylenA\usebox\myboxB}{\usebox\myboxA}%
    \else
        \hskip -0.5\mylenA\rlap{\usebox\myboxA}{\hskip 0.5\mylenA\usebox\myboxB}%
    \fi}
\makeatother

\begin{document}
\title{\bf\Large Tur\'{a}n density of tight cycles minus one edge in the $\ell_2$-norm}
\date{\today}
\author[1]{Levente Bodn\'ar\thanks{Research supported by ERC Advanced Grant 101020255. Email: \texttt{bodnalev@gmail.com}}}
\author[2]{Jinghua Deng\thanks{Email: \texttt{Jinghua\_deng@163.com}}}
\author[2]{Jianfeng Hou\thanks{Research supported by National Key R\&D Program of China (Grant No. 2023YFA1010202), the Central Guidance on Local Science and Technology Development Fund of Fujian Province (Grant No. 2023L3003). Email: \texttt{jfhou@fzu.edu.cn}}}
\author[3]{Xizhi Liu\thanks{Email: \texttt{xizhi.liu.ac@gmail.com}}}
\author[2]{Hongbin Zhao\thanks{Email: \texttt{hbzhao2024@163.com}}}
\affil[1]{Mathematics Institute and DIMAP, University of Warwick, Coventry, CV4 7AL, UK}
\affil[2]{Center for Discrete Mathematics, Fuzhou University, Fujian, 350003, China}
\affil[3]{School of Mathematical Sciences, 
            University of Science and Technology of China,
            Hefei, Anhui, 230026, China}
\maketitle
\begin{abstract}
      The $3$-uniform tight $\ell$-cycle minus one edge $C_{\ell}^{3-}$ is the $3$-graph on $\ell$ vertices consisting of $\ell-1$ consecutive triples in the cyclic order. 
      We show that for every integer $\ell \ge 5$ satisfying $\ell\not\equiv 0\pmod3$, every $C_{\ell}^{3-}$-free $3$-graph whose $\ell_2$-norm, that is,  the sum of codegree squares, is close to the maximum must be structurally close to the iterative blowup of a single triple.
      This confirms a conjecture of Balogh--Clemen--Lidick\'{y}~{\cite[Conjecture~3.5]{BCL22a}} in a stronger form.
\end{abstract}

\section{Introduction}\label{SEC:Introduction}
Given an integer $r\ge 2$, an \textbf{$r$-uniform hypergraph} (henceforth an \textbf{$r$-graph}) $\mathcal{H}$ is a collection of $r$-subsets of some set $V$. 
We call $V$ the \textbf{vertex set} of $\mathcal{H}$ and denote it by $V(\mathcal{H})$. 
The size of $V(\mathcal{H})$ is denoted by $v(\mathcal{H})$.
We identify a hypergraph $\mathcal{H}$ with its set of edges, and hence, $|\mathcal{H}|$ represents the number of edges in $\mathcal{H}$. 

Given an $r$-graph $\mathcal{H}$, the \textbf{shadow} of $\mathcal{H}$ is defined as 
\begin{align*}
    \partial\mathcal{H}
    \coloneqq \left\{e\in \binom{V(\mathcal{H})}{r-1} \colon \text{there exists an edge $E\in \mathcal{H}$ containing $e$} \right\}. 
\end{align*}

For an $(r-1)$-set $T \subseteq V(\mathcal{H})$, the \textbf{neighborhood} of $T$ in $\mathcal{H}$ is defined as 
\begin{align*}
    N_{\mathcal{H}}(T) 
    \coloneqq \left\{v\in V(\mathcal{H}) \colon T \cup \{v\} \in \mathcal{H}\right\}, 
\end{align*}
and the \textbf{codegree} of $T$ is $d_{\mathcal{H}}(T) \coloneqq |N_{\mathcal{H}}(T)|$. 

Following the definitions in~\cite{BCL22b,CILLP24}, for every real number $p \ge 1$, the \textbf{$\ell_p$-norm} of an $r$-graph $\mathcal{H}$ is defined as 
\begin{align*}
    \norm{\mathcal{H}}_{p}
    \coloneqq \sum_{e\in \partial\mathcal{H}} d^{p}_{\mathcal{H}}(e), 
\end{align*}
where $d^{p}_{\mathcal{H}}(e)$ is an abbreviation for $\left(d_{\mathcal{H}}(e)\right)^{p}$. 
Note that for $p=1$, we have $\norm{\mathcal{H}}_{1} = r \cdot |\mathcal{H}|$.

 
Given a family $\mathcal{F}$ of $r$-graphs, an $r$-graph $\mathcal{H}$ is \textbf{$\mathcal{F}$-free}
if it does not contain any member of $\mathcal{F}$ as a subgraph.
The \textbf{$\ell_p$-norm Tur\'{a}n number} $\mathrm{ex}_{\ell_p}(n, \mathcal{F})$ of $\mathcal{F}$ is the maximum $\ell_p$-norm of an $\mathcal{F}$-free $r$-graph $\mathcal{H}$ on $n$ vertices. 
The \textbf{$\ell_p$-norm Tur\'{a}n density} of $\mathcal{F}$ is defined as 
\begin{align*}
    \pi_{\ell_p}(\mathcal{F})
    \coloneqq \lim_{n \to \infty} \frac{\mathrm{ex}_{\ell_p}(n, \mathcal{F})}{n^{r-1+p}}. 
\end{align*}
The existence of this limit follows from a simple averaging argument which can be found in such as~\cite{KNS64},~{\cite[Proposition~1.8]{BCL22b}}, and~{\cite[Proposition~2.2]{CILLP24}}.

Recall that the ordinary Tur\'{a}n number $\mathrm{ex}(n, \mathcal{F})$ of $\mathcal{F}$ is the maximum number of edges in an $n$-vertex $\mathcal{F}$-free $r$-graph.
The corresponding Tur\'{a}n density is defined as $\pi(\mathcal{F}) \coloneqq \lim_{n\to \infty}\mathrm{ex}(n, \mathcal{F})/\binom{n}{r}$. 
Note that the $\ell_1$-norm Tur\'{a}n number of $\mathcal{F}$ is simply $r$ times its ordinary Tur\'{a}n number, and $\pi(\mathcal{F}) = (r-1)! \cdot \pi_{\ell_1}(\mathcal{F})$.  

For $r=2$ (i.e., graphs), the value of $\pi(\mathcal{F})$ is well understood thanks to the celebrated general theorem of Erd\H{o}s--Stone~\cite{ES46} (see also~\cite{ES66}), which extends Tur\'{a}n's seminal theorem on complete graphs~\cite{Tur41} to arbitrary graph families.
The study of $\pi_{\ell_p}(\mathcal{F})$ for $p > 1$ was initiated by Caro--Yuster~\cite{CY00,CY00arxiv}, and Erd\H{o}s--Stone-type results in this setting were later obtained by Bollob\'{a}s--Nikiforov~\cite{BN12}. 

For $r \ge 3$, determining $\pi(\mathcal{F})$ is already notoriously difficult in general, despite significant effort devoted to this area. 
For results up to~2011, we refer the reader to the excellent survey by Keevash~\cite{Kee11}. 
Very recently, Balogh--Clemen--Lidick\'{y}~\cite{BCL22a,BCL22b} initiated the study of $\pi_{\ell_p}(\mathcal{F})$ for hypergraph families. 
Among their many results, they determined the values of $\pi_{\ell_2}(K_{4}^{3})$ and $\pi_{\ell_2}(K_{5}^{3})$, utilizing computer-assisted flag algebra computations, a powerful tool first introduced by~\cite{Raz07}. 
These results are particularly interesting given the notorious difficulty of determining $\pi(K_{\ell}^{3})$ for any $\ell \ge 4$, a problem originally posed by Tur\'{a}n~\cite{Tur41}. 
The $\ell_p$-norm Tur\'{a}n problems for hypergraph have since been explored more systematically in recent works such as~\cite{CL24,CILLP24,GLMP24}.

In this paper, we focus on the following classical object in hypergraph Tur\'an theory. 
For an integer $\ell \ge4$, the \textbf{tight $\ell$-cycle} $C_{\ell}^{3}$ is the $3$-graph on $[\ell]$ with edge set 
\begin{align*}
    \big\{ \{1,2,3\},~\cdots,~\{\ell-2, \ell-1, \ell\},~\{\ell-1, \ell, 1\},~\{\ell, 1,2\} \big\},
\end{align*}
that is, we take all consecutive triples in the cyclic order on $[\ell]$. The \textbf{tight $\ell$-cycle minus one edge} $C_{\ell}^{3-}$ is the $3$-graph on $[\ell]$ with edge set 
\begin{align*}
    \big\{ \{1,2,3\},~\cdots,~\{\ell-2, \ell-1, \ell\},~\{\ell-1, \ell, 1\} \big\},
\end{align*}
that is, $C_{\ell}^{3-}$ is obtained from $C_{\ell}^{3}$ by removing one edge. 

If $\ell \equiv 0 \pmod{3}$ (i.e.\ $\ell$ is divisible by $3$), then  both $C_{\ell}^{3}$ and $C_{\ell}^{3-}$ are $3$-partite and thus it holds that $\pi(C_{\ell}^{3}) = \pi(C_{\ell}^{3-})=0$ by the classical general result of Erd\H{o}s~\cite{Erd64KST}.
Partially inspired by the method used by Kam{\v c}ev--Letzter--Pokrovskiy~\cite{KLP24} for the Tur\'{a}n problem of $C_{\ell}^{3}$, Balogh--Luo~\cite{BL24tightcycle} proved that $\pi(C_{\ell}^{3-}) = 1/4$ for all sufficiently large $\ell$ satisfying $\ell \not\equiv 0 \pmod{3}$. 
Very recently, Lidick\'y--Mattes--Pfender~\cite{LMP24}, and independently Bodn\'ar--Le\'on--Liu--Pikhurko~\cite{BLLP24} proved that $\pi(C_{\ell}^{3-}) = 1/4$ for every $\ell \ge 5$ satisfying $\ell \not\equiv 0 \pmod{3}$.

Recall the following $C_{\ell}^{3-}$-free construction from~\cite{MPS11}. 
For $n \in \{0, 1,2\}$, the $n$-vertex $T_{\mathrm{rec}}$-construction is the empty $3$-graph on $n$ vertices. 
For $n \ge 3$, an $n$-vertex $3$-graph $\mathcal{H}$ is a \textbf{$T_{\mathrm{rec}}$-construction} if there exists a partition $V_1 \cup V_2 \cup V_3 = V(\mathcal{H})$ into non-empty parts such that $\mathcal{H}$ is obtained from $\mathcal{K}[V_1,V_2,V_3]$, the complete $3$-partite $3$-graph with parts $V_1, V_2, V_3$, by adding a copy of $T_{\mathrm{rec}}$-construction into each $V_i$ for $i \in [3]$. 
A $3$-graph is a \textbf{$T_{\mathrm{rec}}$-subconstruction} if it a subgraph of some $T_{\mathrm{rec}}$-construction.
It is easy to see that every $T_{\mathrm{rec}}$-subconstruction is $C_{\ell}^{3-}$-free for every integer $\ell \not\equiv 0 \pmod{3}$.

Let $t^{rec}_{2}(n)$ denote the maximum $\ell_{2}$-norm of an $n$-vertex $T_{\mathrm{rec}}$-construction. 
According to the definition, for every $n\geq3$, we have:
\begin{align*}
    t^{rec}_{2}(n)
    = \max\left\{n_1n_2n_3n + \sum_{i\in[3]}t^{rec}_{2}(n_i) \colon \text{$n_1+n_2+n_3 = n$ and $n_i \ge 1$ for $i \in [3]$} \right\}.
\end{align*}
Simple calculations show that $t^{rec}_{2}(n)=(1/26 + o(1))n^4$.  

In~{\cite[Conjecture~3.5]{BCL22a}}, Balogh--Clemen--Lidick{\'y} conjectured that the lower bound $1/26$, given by the $3$-partite recursive constructions described above, is also the upper bound for $\pi_{\ell_2}(C_{5}^{3-})$, that is, $\pi_{\ell_2}(C_{5}^{3-}) = 1/26$. 
The main result of this work is a confirmation of their conjecture, along with the establishment of an Erd\H os--Simonovits--type stability theorem~\cite{Erdos67a,Sim68} for $C_{\ell}^{3-}$-free $3$-graphs in the $\ell_2$-norm.

\begin{theorem}\label{THM:Main-result-C5Minus-L2Norm}
    Let $\ell \ge 5$ be an integer satisfying $\ell \not\equiv 0 \pmod{3}$. 
    Then $\pi_{\ell_2}(C_{\ell}^{3-}) = 1/26$. 
    Moreover, for every $\varepsilon>0$ there exist $\delta$ and $n_0$ such that the following holds for every $n\ge n_0$. 
    Suppose that $\mathcal{H}$ is an $n$-vertex $C_{\ell}^{3-}$-free $3$-graph with $\norm{\mathcal{H}}_2 \ge (1/26-\delta)n^4$. 
    Then $\mathcal{H}$ is a $T_{\mathrm{rec}}$-subconstruction after removing at most $\varepsilon n^3$ edges.
\end{theorem}

Our proofs crucially use the flag algebra machinery developed by Razborov~\cite{Raz07} and are computer-assisted. 
More specifically, we adopt the strategy used in the previous work~\cite{BLLP24} for determining the Tur\'{a}n density of $C_{5}^{3-}$, which is inspired by the strategy of Balogh--Hu--Lidick{\'y}--Pfender used in~\cite{BHLP16C5} for determining the inducibility of the $5$-cycle $C_5$ (where asymptotically extremal graphs are also obtained via a recursive construction). 

Some standard reductions (see the discussion in Section~\ref{SEC:Prelim}) show that Theorem~\ref{THM:Main-result-C5Minus-L2Norm} is equivalent to Theorem~\ref{THM:C5Minus-S2}, which concerns the density of $\mathbb{S}_{2}$ (the unique $2$-edge $3$-graph on $4$ vertices) in $\{K_{4}^{3-}, C_{5}^{3-}\}$-free $3$-graphs.
Note that due to the $\{K_{4}^{3-}, C_{5}^{3-}\}$-freeness, every copy of $\mathbb{S}_{2}$ must now be an induced copy.
Our proof for Theorem~\ref{THM:Main-result-C5Minus-L2Norm} is based on the following crucial claims about a $\{K_{4}^{3-}, C_{5}^{3-}\}$-free $3$-graph $\mathcal{H}$ with $n \to \infty$ vertices and $\mathbb{S}_{2}$-density close to the maximum value $6/13$.
\begin{enumerate}[label=(\roman*)]
    \item\label{it:A} Proposition~\ref{PROP:max-cut-ratio-lower-bound} shows that there exists a vertex partition $V_1\cup V_2\cup V_3=V(\mathcal{H})$ such that $|\mathcal{H}\cap \mathcal{K}[V_1,V_2,V_3]| \ge (0.198... + o(1))\binom{n}{3}$ (which is not too far from the upper bound $(n/3)^3 = (0.222...+o(1)) \binom{n}{3}$).
    By moving vertices between parts, we may further assume that this partition is \textbf{locally maximal}, meaning that no single-vertex move between parts increases the number of \textbf{transversal edges}, that is, edges in $\mathcal{H}\cap \mathcal{K}[V_1,V_2,V_3]$. 
    It is worth noting that the locally maximal property is crucial for the claims that follow.
    \item\label{it:B} Proposition~\ref{PROP:B-vs-075M} shows that when comparing $\mathcal{H} \setminus \bigcup_{i=1}^3 \mathcal{H}[V_i]$ with the complete $3$-partite $3$-graph $\mathcal{K}[V_1,V_2,V_3]$, the number of additional edges (referred to as \textbf{bad edges}) is at most $0.75$ times the number of missing edges. The constant $0.75$ is crucial for our proof, in previous work~{\cite[Proposition~3.3]{BLLP24}}, this constant was $0.99$.
    \item\label{it:C} Proposition~\ref{PROP:BS2-vs-09MS2} shows that, ignoring the copies of $\mathbb{S}_{2}$ contained within the parts $V_i$ for $i \in [3]$, the number of copies of $\mathbb{S}_{2}$ in $\mathcal{H}$ is strictly less than the number in $\mathcal{K}[V_1,V_2,V_3]$, unless $\mathcal{H} \setminus \bigcup_{i=1}^3 \mathcal{H}[V_i] = \mathcal{K}[V_1,V_2,V_3]$. 
\end{enumerate}

Thus have identified a top-level partition such that, ignoring copies of $\mathbb{S}_{2}$ inside parts, $\mathcal{H}$ does not perform better than a copy of $T_{\mathrm{rec}}$ that uses the same parts. 
We can recursively apply this result to each part $\mathcal{H}[V_i]$ as long $|V_i|$ is sufficiently large. 
Now, routine calculations imply that the number of $\mathbb{S}_{2}$ in $\mathcal{H}$ is at most $(6/13 + o(1))\binom{n}{4}$.

In Section~\ref{SEC:Prelim}, we present the necessary definitions and preliminary results. 
The proofs of Theorem~\ref{THM:C5Minus-S2}~\ref{THM:C5Minus-S2-a} and~\ref{THM:C5Minus-S2-b}, which together imply Theorem~\ref{THM:Main-result-C5Minus-L2Norm}, are presented in Sections~\ref{SEC:proof-S2-density} and~\ref{SEC:proof-S2-stability}, respectively. 
Finally, Section~\ref{SEC:Remark} contains some concluding remarks.

\section{Preliminaries}\label{SEC:Prelim}
For pairwise disjoint sets $V_1, \ldots, V_{\ell}$, we use $\mathcal{K}[V_1, \ldots, V_{\ell}]$ to denote the complete $\ell$-partite $\ell$-graph with parts $V_1, \ldots, V_{\ell}$. In particular, $\mathcal{K}[V_1, V_2, V_3]$ denote the complete $3$-partite $3$-graph with parts $V_1, V_2, V_3$. 
Denote by $K[V_1, \ldots, V_{\ell}]$ the complete $\ell$-partite $2$-graph with parts $V_1, \ldots, V_{\ell}$.

Given a $3$-graph $\mathcal{H}$, the \textbf{link} of a vertex $v \in V(\mathcal{H})$ is given by 
\begin{align*}
    L_{\mathcal{H}}(v)
    \coloneqq \left\{e \in \partial\mathcal{H} \colon \{v\} \cup e \in \mathcal{H}\right\}.
\end{align*}
The \textbf{degree} of $v$ is $d_{\mathcal{H}}(v) \coloneqq |L_{\mathcal{H}}(v)|$. 

For every integer $k \ge 1$, the $k$-blowup $\mathcal{H}^{(k)}$ of $\mathcal{H}$ is the $3$-graph whose vertex set is the union $\bigcup _{v \in V(\mathcal{H})} U_{v}$ of some disjoint $k $-sets $U_{v}$, one per each vertex $v \in V(\mathcal{H})$, and whose edge set is the union of the complete 3-partite $3$-graphs $\mathcal K[U_{x}, U_{y}, U_{z}]$ over all edges $\{x, y, z\} \in \mathcal{H}$. Informally speaking, $\mathcal{H}^{(k)}$ is obtained from $\mathcal{H}$ by cloning each vertex $k$ times.

Given a $3$-graph $F$ on $k$ vertices, the \textbf{(induced) density} $p(F,\mathcal{H})$ of $F$ in $\mathcal{H}$ is the number of $k$-subsets of $V(\mathcal{H})$ that span a subgraph isomorphic to $F$, divided by $\binom{v(\mathcal{H})}{k}$. 
When $F=K_3^3$ is the single edge, we get the \textbf{edge density} $\rho(\mathcal{H}) \coloneqq |\mathcal{H}|/\binom{v(\mathcal{H})}{3}$.

Let $\mathbb{S}_{2}$ be the unique $3$-graph on four vertices with exactly two edges. 
Denote by $\mathcal{N}(\mathbb{S}_2, \mathcal{H})$ the collection of all copies of $\mathbb{S}_{2}$ in a $3$-graph $\mathcal{H}$. 
Let $\mathrm{N}(\mathbb{S}_2, \mathcal{H}) \coloneqq |\mathcal{N}(\mathbb{S}_2, \mathcal{H})|$ denote the number of copies of $\mathbb{S}_2$ in $\mathcal{H}$.
The $\mathbb{S}_{2}$-density of $\mathcal{H}$ is defined as $\rho(\mathbb{S}_{2},\mathcal{H}) \coloneqq \mathrm{N}(\mathbb{S}_2, \mathcal{H})/\binom{v(\mathcal{H})}{4}$. 

For a family $\mathcal{F}$ of $3$-graphs, we define the generalized Tur\'{a}n number and generalized Tur\'{a}n density as follows:
\begin{align*}
    \mathrm{ex}(n,\mathbb{S}_{2}, \mathcal{F})
    & \coloneqq \max\big\{ \mathrm{N}(\mathbb{S}_2, \mathcal{H}) \colon \text{$v(\mathcal{H}) = n$ and $\mathcal{H}$ is $\mathcal{F}$-free} \big\} \quad\text{and}\\[0.3em]
    \pi(\mathbb{S}_2, \mathcal{F})
    & \coloneqq \lim_{n\to \infty} {\mathrm{ex}(n,\mathbb{S}_{2}, \mathcal{F})}/{\binom{n}{4}}. 
\end{align*}
Since for every $3$-graph $\mathcal{H}$, we have 
\begin{align*}
    \mathrm{N}(\mathbb{S}_2, \mathcal{H})
    = \sum_{e\in \partial\mathcal{H}} \binom{d_{\mathcal{H}}(e)}{2}
    =  \sum_{e\in \partial\mathcal{H}} \frac{d_{\mathcal{H}}^{2}(e) - d_{\mathcal{H}}(e)}{2}
    = \frac{\norm{\mathcal{H}}_{2} - 3|\mathcal{H}|}{2}, 
\end{align*}
it follows that for every family $\mathcal{F}$ of $3$-graphs,  
\begin{align}\label{equ:S2-density-vs-L2-density}
    \pi(\mathbb{S}_2, \mathcal{F})
    = 12 \cdot \pi_{\ell_2}(\mathcal{F}). 
\end{align}
Therefore, the $\ell_{2}$-norm Tur\'{a}n problem for $\mathcal{F}$ is asymptotically equivalent to the corresponding generalized Tur\'{a}n problem concerning $\mathbb{S}_{2}$. 

Given two $r$-graphs $\mathcal{H}$ and $\mathcal{G}$, a map $\psi \colon V(\mathcal{H})\to V(\mathcal{G})$ is a \textbf{homomorphism} from $\mathcal{H}$ to $\mathcal{G}$ if $\psi(e) \in \mathcal{G}$ for all $e\in \mathcal{H}$. 
Note that there is a homomorphism from $\mathcal{H}$ to $\mathcal{G}$ iff $\mathcal{H}$ is contained in some blowup of $\mathcal{G}$. 

Let $K_{4}^{3-}$ denote the $3$-graph on $\{1,2,3,4\}$ with edge set 
\begin{align*}
    \big\{\{1,2,3\},~\{1,2,4\},~\{1,3,4\}\big\}. 
\end{align*}
Observe that there exists a homomorphism from $C_{5}^{3-}$ to $K_{4}^{3-}$ (obtained by merging two vertices not contained in any edge), i.e., $C_{5}^{3-}$ is contained in some blowup of $K_{4}^{3-}$. 
Thus, by the Supersaturation Method (see e.g.~\cite[Proposition~1.9]{BCL22b}) and the Hypergraph Removal Lemma (see e.g.~\cite{G07,RNSS05,RS04}), every $n$-vertex $C_{5}^{3-}$-free $3$-graph can be made $K_{4}^{3-}$-free by removing $o(n^3)$ edges (and thus removes only $o(n^4)$ copies of $\mathbb{S}_{2}$). 

Similarly, for every integer $\ell \ge 5$ with $\ell \not\equiv 0 \pmod{3}$, there exists a homomorphism from $C_{\ell}^{3-}$ to $C_{5}^{3-}$ (see e.g.~{\cite[Claim~5.14]{BL24tightcycle}}). 
It follows that every $n$-vertex $C_{\ell}^{3-}$-free $3$-graph can be made $C_{5}^{3-}$-free by removing $o(n^4)$ copies of $\mathbb{S}_{2}$. 

Combining these observations with~\eqref{equ:S2-density-vs-L2-density}, we see that Theorem~\ref{THM:Main-result-C5Minus-L2Norm} reduces to the following result. 

\begin{theorem}\label{THM:C5Minus-S2}
    The following statements hold. 
    \begin{enumerate}[label=(\roman*)]
        \item\label{THM:C5Minus-S2-a} We have $\pi(\mathbb{S}_{2}, \{K_{4}^{3-}, C_{5}^{3-}\}) = 6/13$.
        \item\label{THM:C5Minus-S2-b} For every $\varepsilon>0$ there exist $\delta$ and $n_0$ such that the following holds for every $n\ge n_0$. 
        Suppose that $\mathcal{H}$ is an $n$-vertex $\{K_{4}^{3-}, C_{5}^{3-}\}$-free $3$-graph with $\mathrm{N}(\mathbb{S}_{2}, \mathcal{H}) \ge (6/13 - \delta)n^4$. 
        Then $\mathcal{H}$ is a $T_{\mathrm{rec}}$-subconstruction after removing at most $\varepsilon n^3$ edges.
    \end{enumerate}
\end{theorem}

Recall that the \textbf{standard $2$-dimensional simplex} is defined as  
\begin{align*}
    \Delta^{2}
    \coloneqq \left\{\left(x_{1}, x_{2}, x_{3}\right) \in \mathbb{R}^{3} \colon x_{1}+x_{2}+x_{3}=1 \text{ and } x_{i} \geq 0 \text{ for } i \in[3]\right\}.
\end{align*}
The following inequalities can be verified using elementary analysis.
\begin{fact}\label{FACT:inequalities}
    The following inequalities hold for every $(x_{1}, x_{2}, x_{3}) \in \Delta^{2}$. 
    \begin{enumerate}[label=(\roman*)]
        \item\label{FACT:inequalities-1} Suppose that $\min\{x_1, x_2, x_3\} > 0$. Then 
        \begin{align*}
            \frac{x_1x_2x_3}{1-(x_1^4+x_2^4+x_3^4)}
            \le \frac{1}{26}. 
        \end{align*}
        %
        %
        \item\label{FACT:inequalities-3} Suppose that $\min\{x_1,x_2,x_3\} \ge 1/5$. Then 
        \begin{align*}
            x_1x_2x_3+\frac{x_1^4+x_2^4+x_3^4}{26}
            \le \frac{1}{26}- \frac{1}{15}\sum_{i\in [3]} \left(x_i-\frac{1}{3}\right)^2.
        \end{align*}
    \end{enumerate}
\end{fact}

We will also use the following theorem, which states that the Tur\'{a}n density of $\{K_{4}^{3-}, C_{5}^{3-}\}$ is $1/4$.
In other words, in every $n$-vertex $\{ K_{4}^{3-}, C_{5}^{3-} \}$-free $3$-graph, the ratio of edges to non-edges is at most $1/3 + o(1)$. 
\begin{theorem}[\cite{LMP24,BLLP24}]\label{THM:BLLP24-C5Minus-Turan-Density}
    We have $\pi(\{K_{4}^{3-}, C_{5}^{3-}\}) = 1/4$. 
\end{theorem}

\section{Computer-generated results}\label{SEC:Flag-Algebra}
In this section, we present the results derived by computer using the flag algebra method of Razborov~\cite{Raz07}, which is also described in e.g.~\cite{Razborov10,BaberTalbot11,SFS16,GilboaGlebovHefetzLinialMorgenstein22}.
Since this method is well-known by now, we will be very brief. In particular,  we omit many definitions, referring
the reader to~\cite{Raz07,Razborov10} for any missing notions. Roughly speaking, a \textbf{flag algebra proof using $0$-flags on $m$ vertices} of an upper bound $u\in\mathbb{R}$ on the given objective function $f$ 
consists of an identity 
\begin{align*}
    u-f(\mathcal{H})
    = \mathrm{SOS}+\sum_{F\in\mathcal{F}_m^0}c_F \cdot p(F,\mathcal{H})+o(1),
\end{align*}
which is asymptotically true for any admissible $\mathcal{H}$ with $|V(\mathcal{H})|\to\infty$, where the $\mathrm{SOS}$-term can be represented as a sum of squares (as described e.g. in~{\cite[Section~3]{Razborov10}}), each coefficient $c_F\in\mathbb{R}$ is non-negative, and $\mathcal{F}_m^0$ consists of isomorphism types of \textbf{$0$-flags} (unlabelled $3$-graphs) with $m$ vertices. If $f(\mathcal{H})$ can be represented as a linear combination of the densities of members of $\mathcal{F}_m^0$ in $\mathcal{H}$ then finding the smallest possible $u$ amounts to solving a semi-definite program (SDP) with $|\mathcal{F}_m^0|$ linear constraints (so we write the size of $\mathcal{F}_m^0$ in each case to give the reader some idea of the size of the programs that we had to solve).

We formed the corresponding SDPs and then analyzed the solutions returned by computer, using a modified version of the SageMath package. 
This package is still under development, for short guide on how to install it and its current functionality can be found in the GitHub repository \href{https://github.com/bodnalev/sage}{\url{https://github.com/bodnalev/sage}}. The calculations used for this paper and the generated certificates can be found in \href{https://drive.google.com/drive/folders/175ciwbvhYCeCQmub2mVzslRxopIHZAl6?usp=share_link}{\url{https://drive.google.com/drive/folders/175ciwbvhYCeCQmub2mVzslRxopIHZAl6?usp=share_link}}.
We did not try to reduce the set of the used types needed for the proofs, although we did use the (standard) observation of Razborov~~{\cite[Section~4]{Razborov10}} that each unknown SDP matrix can be assumed to consist of $2$ blocks (namely, the invariant and anti-invariant parts).


For an $n$-vertex $3$-graph $\mathcal{H}$ define the \textbf{max-cut ratio} to be 
\begin{align*}
    \mu_{\mathcal{H}}(V_1, V_2, V_3)
    \coloneqq \frac{6}{n^{3}} \cdot \max\big\{\left|\mathcal{H} \cap \mathcal{K}\left[V_{1}, V_{2}, V_{3}\right]\right| \colon V_{1}, V_{2}, V_{3} \text{ form a partition of } V(\mathcal{H}) \big\}.
\end{align*}

The following result shows that every $\{K_{4}^{3-}, C_{5}^{3-}\}$-free $3$-graph $\mathcal{H}$ on $n$ vertices with large  $\mathbb{S}_{2}$-density must have a large max-cut ratio, that is, a large $3$-partite subgraph.  

\begin{proposition}\label{PROP:max-cut-ratio-lower-bound}\label{pr2}
     Suppose that $\mathcal{H}$ is an $n$-vertex $\{K_{4}^{3-}, C_{5}^{3-}\}$-free $3$-graph with $\rho(\mathbb{S}_{2}, \mathcal{H}) \ge \beta_{\ref{pr2}}$, where $\beta_{\ref{pr2}} \coloneqq 6/13 - 10^{-6}$. 
     Then $\mathcal{H}$ has the max-cut ratio $\mu_{\mathcal{H}}(V_1, V_2, V_3)$ at least $\alpha_{\ref{pr2}}$, where 
     \begin{align}\label{equ:alpha-32}
        \alpha_{\ref{pr2}}
        \coloneqq \frac{675468913113}{3407872000000} 
        = 0.198208....
    \end{align}
\end{proposition}

Note that Proposition~\ref{PROP:max-cut-ratio-lower-bound} is nearly identical to~{\cite[Proposition~3.2]{BLLP24}}, except that we replace the edge density assumption $\rho(\mathcal{H}) \ge 1/4 - 10^{-6}$ with the $\mathbb{S}_{2}$-density assumption $\rho(\mathbb{S}_{2}, \mathcal{H}) \ge 6/13 - 10^{-6}$.
Since the proofs are essentially the same, we omit it here.

For a partition $V_1\cup V_2\cup V_3 = V \coloneqq V(\mathcal{H})$ of the vertex set of $\mathcal{H}$, define the sets of bad edges and missing triples: 
\begin{align}
    B_{\mathcal{H}}\left(V_{1}, V_{2}, V_{3}\right)
    & \coloneqq \left\{e \in \mathcal{H} \colon \left\{\left|e \cap V_{1}\right|,\left|e \cap V_{2}\right|,\left|e \cap V_{3}\right|\right\} = \{0,1,2\}\right\},\label{equ:DEF-BH}\\[0.5em]
    M_{\mathcal{H}}\left(V_{1}, V_{2}, V_{3}\right)
    & \coloneqq \left\{e \in \binom{V}{3} \setminus \mathcal{H} \colon \left\{\left|e \cap V_{1}\right|,\left|e \cap V_{2}\right|,\left|e \cap V_{3}\right|\right\} = \{1,1,1\}\right\}. \label{equ:DEF-MH}
\end{align}
Note that the edges inside a part are not classified as bad or missing.

Similarly, define the sets of bad and missing $\mathbb{S}_{2}$'s:
\begin{align}
   B\mathbb{S}_{2}^{\mathcal{H}}(V_1,V_2,V_3)
   & \coloneqq \left\{ S \in \mathcal{N}(\mathbb{S}_{2}, \mathcal{H})  \colon \text{$S$ contains some edge in $B_{\mathcal{H}}\left(V_{1}, V_{2}, V_{3}\right)$}\right\}, \label{equ:DEF-BS2}\\[0.5em]
   M\mathbb{S}_{2}^{\mathcal{H}}(V_1,V_2,V_3)
   & \coloneqq \mathcal{N}(\mathbb{S}_{2}, \mathcal{K}[V_1, V_2, V_3]) \setminus \mathcal{N}(\mathbb{S}_{2}, \mathcal{H}\cap \mathcal{K}[V_1, V_2, V_3])). \label{equ:DEF-MS2}
\end{align}
That is, ignoring the copies of $\mathbb{S}_{2}$ that are entirely contained in $V_i$ for $i \in [3]$, we define a bad $\mathbb{S}_{2}$ as one that includes at least one bad edge from $B_{\mathcal{H}}\left(V_{1}, V_{2}, V_{3}\right)$, and a missing $\mathbb{S}_{2}$ as one that appears in $\mathcal{K}[V_1, V_2, V_3]$ but not in $\mathcal{H}$. 
Similarly, the copies of $\mathbb{S}_{2}$ inside a part are not classified as bad or missing.

For convenience, we will omit $(V_1, V_2, V_3)$ and the subscript/superscript $\mathcal{H}$ when they are clear from context.

In~{\cite[Proposition~3.3]{BLLP24}}, it was shown that if  $V_1\cup V_2\cup V_3 = V(\mathcal{H})$ is a locally maximal partition of a $\{K_{4}^{3-}, C_{5}^{3-}\}$-free $3$-graph $\mathcal{H}$ such that $\mu_{\mathcal{H}}(V_1, V_2, V_3) \ge 0.19$, then $|B| - 0.99 |M| \le 0$. 
The following result improves the constant $0.99$ to $0.75$ under the slightly stronger assumption that $\mu_{\mathcal{H}}(V_1, V_2, V_3) \ge 0.198$. 
Since the proofs are the same, we omit it here. 

\begin{proposition}\label{PROP:B-vs-075M}\label{pr3}
    Suppose that $\mathcal{H}$ is an $n$-vertex $\{K_{4}^{3-}, C_{5}^{3-}\}$-free  $3$-graph and $V(\mathcal{H})=V_{1} \cup V_{2} \cup V_{3}$ is a locally maximal partition with $\mu_{\mathcal{H}}(V_1, V_2, V_3) \ge 0.198$. 
    Then with $B = B_{\mathcal{H}}(V_{1}, V_{2}, V_{3})$ and $M = M_{\mathcal{H}}(V_{1}, V_{2}, V_{3})$ as defined in~\eqref{equ:DEF-BH} and~\eqref{equ:DEF-MH}, we have 
    \begin{align*}
        |B|-\frac{3}{4}|M| \leq 0.
    \end{align*}
\end{proposition}

\begin{figure}[H]
\centering
\tikzset{every picture/.style={line width=1pt}} 
\begin{tikzpicture}[x=0.75pt,y=0.75pt,yscale=-1,xscale=1,line join=round]
\draw   (182.39,84.6) .. controls (182.39,67.15) and (196.79,53) .. (214.54,53) .. controls (232.3,53) and (246.69,67.15) .. (246.69,84.6) .. controls (246.69,102.06) and (232.3,116.21) .. (214.54,116.21) .. controls (196.79,116.21) and (182.39,102.06) .. (182.39,84.6) -- cycle ;
\draw   (127.88,172.2) .. controls (127.88,154.74) and (142.28,140.59) .. (160.03,140.59) .. controls (177.79,140.59) and (192.18,154.74) .. (192.18,172.2) .. controls (192.18,189.65) and (177.79,203.8) .. (160.03,203.8) .. controls (142.28,203.8) and (127.88,189.65) .. (127.88,172.2) -- cycle ;
\draw   (236.9,172.2) .. controls (236.9,154.74) and (251.3,140.59) .. (269.05,140.59) .. controls (286.81,140.59) and (301.2,154.74) .. (301.2,172.2) .. controls (301.2,189.65) and (286.81,203.8) .. (269.05,203.8) .. controls (251.3,203.8) and (236.9,189.65) .. (236.9,172.2) -- cycle ;
\draw [pattern={Lines[angle=30,distance={2.3pt},line width=0.8pt]}, pattern color = uuuuuu, fill opacity=0.4]  (202.2,81.8) .. controls (206.2,123.8) and (222.2,152.8) .. (269.05,172.2) -- (160.03,172.2) .. controls (191.2,134.8) and (193.2,123.8) .. (202.2,81.8) ;
\draw  [pattern={dots}, pattern color = uuuuuu, fill opacity=0.4, dashed]  (227.2,82.8) .. controls (231.2,124.8) and (239.2,141.8) .. (269.05,172.2) -- (160.03,172.2) .. controls (208.2,148.8) and (218.2,124.8) .. (227.2,82.8) ;
\draw [fill=uuuuuu] (227.2,82.8) circle (1.3pt);
\draw [fill=uuuuuu] (269.05,172.2) circle (1.3pt);
\draw [fill=uuuuuu] (160.03,172.2) circle (1.3pt);
\draw [fill=uuuuuu] (202.2,81.8) circle (1.3pt);
\draw   (410.39,81.6) .. controls (410.39,64.15) and (424.79,50) .. (442.54,50) .. controls (460.3,50) and (474.69,64.15) .. (474.69,81.6) .. controls (474.69,99.06) and (460.3,113.21) .. (442.54,113.21) .. controls (424.79,113.21) and (410.39,99.06) .. (410.39,81.6) -- cycle ;
\draw   (355.88,169.2) .. controls (355.88,151.74) and (370.28,137.59) .. (388.03,137.59) .. controls (405.79,137.59) and (420.18,151.74) .. (420.18,169.2) .. controls (420.18,186.65) and (405.79,200.8) .. (388.03,200.8) .. controls (370.28,200.8) and (355.88,186.65) .. (355.88,169.2) -- cycle ;
\draw   (464.9,169.2) .. controls (464.9,151.74) and (479.3,137.59) .. (497.05,137.59) .. controls (514.81,137.59) and (529.2,151.74) .. (529.2,169.2) .. controls (529.2,186.65) and (514.81,200.8) .. (497.05,200.8) .. controls (479.3,200.8) and (464.9,186.65) .. (464.9,169.2) -- cycle ;
\draw  [pattern={dots}, pattern color = uuuuuu, fill opacity=0.4, dashed]  (430.2,78.8) .. controls (434.2,120.8) and (450.2,149.8) .. (497.05,169.2) -- (388.03,169.2) .. controls (419.2,131.8) and (421.2,120.8) .. (430.2,78.8) ;
\draw  [pattern={dots}, pattern color = uuuuuu, fill opacity=0.4, dashed]  (455.2,79.8) .. controls (459.2,121.8) and (467.2,138.8) .. (497.05,169.2) -- (388.03,169.2) .. controls (436.2,145.8) and (446.2,121.8) .. (455.2,79.8) ;
\draw [fill=uuuuuu] (430.2,78.8) circle (1.3pt);
\draw [fill=uuuuuu] (455.2,79.8) circle (1.3pt);
\draw [fill=uuuuuu] (497.05,169.2) circle (1.3pt);
\draw [fill=uuuuuu] (388.03,169.2) circle (1.3pt);
\end{tikzpicture}
\caption{Two types of missing $\mathbb{S}_{2}$ (there may be other edges not indicated in the figure).}
\label{fig:MS2}
\end{figure}
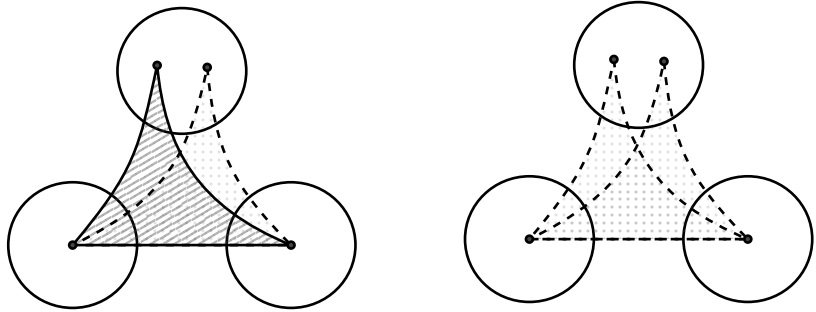

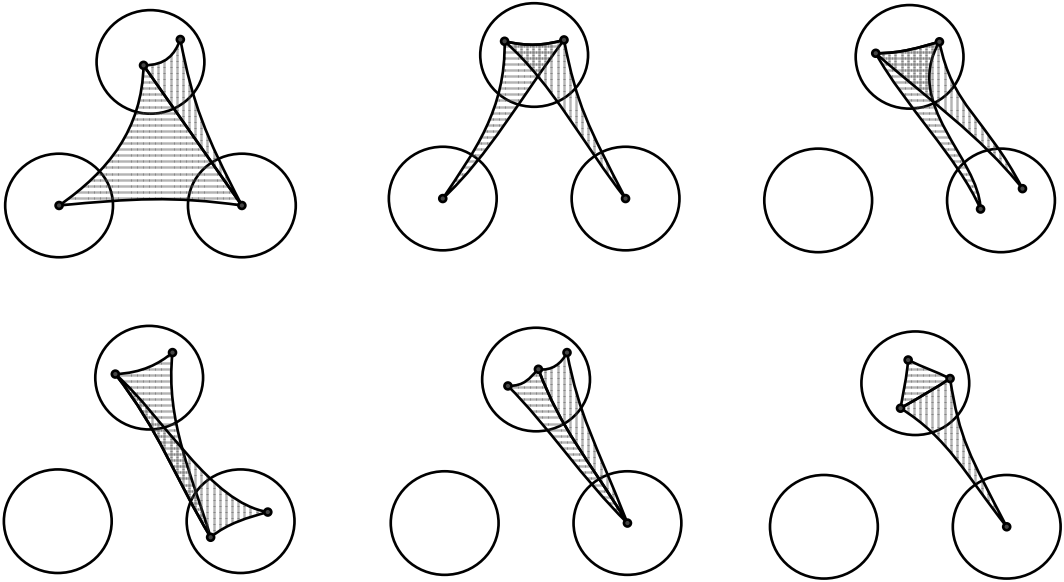
\begin{figure}[H]
\centering
\tikzset{every picture/.style={line width=1pt}} 
\begin{tikzpicture}[x=0.75pt,y=0.75pt,yscale=-1,xscale=1,line join=round]
\draw   (111.99,68.53) .. controls (111.99,54.15) and (124.03,42.48) .. (138.89,42.48) .. controls (153.75,42.48) and (165.8,54.15) .. (165.8,68.53) .. controls (165.8,82.92) and (153.75,94.58) .. (138.89,94.58) .. controls (124.03,94.58) and (111.99,82.92) .. (111.99,68.53) -- cycle ;
\draw   (66.37,140.74) .. controls (66.37,126.35) and (78.41,114.69) .. (93.27,114.69) .. controls (108.13,114.69) and (120.18,126.35) .. (120.18,140.74) .. controls (120.18,155.12) and (108.13,166.79) .. (93.27,166.79) .. controls (78.41,166.79) and (66.37,155.12) .. (66.37,140.74) -- cycle ;
\draw   (157.61,140.74) .. controls (157.61,126.35) and (169.66,114.69) .. (184.52,114.69) .. controls (199.38,114.69) and (211.42,126.35) .. (211.42,140.74) .. controls (211.42,155.12) and (199.38,166.79) .. (184.52,166.79) .. controls (169.66,166.79) and (157.61,155.12) .. (157.61,140.74) -- cycle ;
\draw  [pattern={Lines[angle=0,distance={2.3pt},line width=0.8pt]}, pattern color = uuuuuu, fill opacity=0.4]  (93.27,140.74) .. controls (117.78,123.57) and (134.73,102.38) .. (135.43,70.24) -- (184.52,140.74) .. controls (150.26,135.88) and (129.78,137.25) .. (93.27,140.74) ;
\draw  [pattern={Lines[angle=90,distance={2.3pt},line width=0.8pt]}, pattern color = uuuuuu, fill opacity=0.4]  (184.52,140.74) -- (135.43,70.24) .. controls (146.73,70.24) and (150.97,64.77) .. (153.8,57.25).. controls  (160.15,90.07) and (165.8,104.43) .. (184.52,140.74);
\draw [fill=uuuuuu] (184.52,140.74) circle (1.3pt);
\draw [fill=uuuuuu] (135.43,70.24) circle (1.3pt);
\draw [fill=uuuuuu] (153.8,57.25) circle (1.3pt);
\draw [fill=uuuuuu] (93.27,140.74) circle (1.3pt);
\draw   (303.5,65.05) .. controls (303.5,50.66) and (315.55,39) .. (330.41,39) .. controls (345.27,39) and (357.32,50.66) .. (357.32,65.05) .. controls (357.32,79.44) and (345.27,91.1) .. (330.41,91.1) .. controls (315.55,91.1) and (303.5,79.44) .. (303.5,65.05) -- cycle ;
\draw   (257.88,137.25) .. controls (257.88,122.87) and (269.93,111.2) .. (284.79,111.2) .. controls (299.65,111.2) and (311.7,122.87) .. (311.7,137.25) .. controls (311.7,151.64) and (299.65,163.3) .. (284.79,163.3) .. controls (269.93,163.3) and (257.88,151.64) .. (257.88,137.25) -- cycle ;
\draw   (349.13,137.25) .. controls (349.13,122.87) and (361.17,111.2) .. (376.03,111.2) .. controls (390.89,111.2) and (402.94,122.87) .. (402.94,137.25) .. controls (402.94,151.64) and (390.89,163.3) .. (376.03,163.3) .. controls (361.17,163.3) and (349.13,151.64) .. (349.13,137.25) -- cycle ;
\draw  [pattern={Lines[angle=0,distance={2.3pt},line width=0.8pt]}, pattern color = uuuuuu, fill opacity=0.4]  (284.79,137.25) .. controls (300.2,113.99) and (316.37,92.54) .. (315.65,58.14) .. controls (327.66,60.88) and (335.43,60.2) .. (345.31,57.46)  .. controls   (323.42,86.86) and (307.88,115.58) .. (284.79,137.25);
\draw [pattern={Lines[angle=90,distance={2.3pt},line width=0.8pt]}, pattern color = uuuuuu, fill opacity=0.4]   (376.03,137.25) .. controls (357.32,100.95) and (351.67,90.28) .. (345.31,57.46) .. controls (335.43,60.2) and (327.66,60.88) .. (315.65,58.14) .. controls (340.15,80.04) and (355.37,110.82) .. (376.03,137.25);
\draw [fill=uuuuuu] (376.03,137.25) circle (1.3pt);
\draw [fill=uuuuuu] (345.31,57.46) circle (1.3pt);
\draw [fill=uuuuuu] (315.65,58.14) circle (1.3pt);
\draw [fill=uuuuuu] (284.79,137.25) circle (1.3pt);
\draw   (490.91,65.98) .. controls (490.91,51.6) and (502.96,39.93) .. (517.82,39.93) .. controls (532.68,39.93) and (544.72,51.6) .. (544.72,65.98) .. controls (544.72,80.37) and (532.68,92.03) .. (517.82,92.03) .. controls (502.96,92.03) and (490.91,80.37) .. (490.91,65.98) -- cycle ;
\draw   (445.29,138.19) .. controls (445.29,123.8) and (457.34,112.14) .. (472.2,112.14) .. controls (487.06,112.14) and (499.1,123.8) .. (499.1,138.19) .. controls (499.1,152.57) and (487.06,164.24) .. (472.2,164.24) .. controls (457.34,164.24) and (445.29,152.57) .. (445.29,138.19) -- cycle ;
\draw   (536.53,138.19) .. controls (536.53,123.8) and (548.58,112.14) .. (563.44,112.14) .. controls (578.3,112.14) and (590.35,123.8) .. (590.35,138.19) .. controls (590.35,152.57) and (578.3,164.24) .. (563.44,164.24) .. controls (548.58,164.24) and (536.53,152.57) .. (536.53,138.19) -- cycle ;
\draw  [pattern={Lines[angle=0,distance={2.3pt},line width=0.8pt]}, pattern color = uuuuuu, fill opacity=0.4]  (500.92,64.18) .. controls (514.24,64.18) and (522.83,61.13) .. (532.72,58.39) .. controls (514.24,90.3) and  (552.29,116.42) .. (553.25,142.53) .. controls (541.83,116.42) and (515.19,90.3) .. (500.92,64.18) ;
\draw  [pattern={Lines[angle=90,distance={2.3pt},line width=0.8pt]}, pattern color = uuuuuu, fill opacity=0.4]  (500.92,64.18) .. controls (514.24,64.18) and (522.83,61.13) .. (532.72,58.39) .. controls (539.08,91.21) and (555.46,95.97) .. (574.17,132.27) .. controls (553.52,105.84) and (525.42,86.08) .. (500.92,64.18);
\draw [fill=uuuuuu] (500.92,64.18) circle (1.3pt);
\draw [fill=uuuuuu] (532.72,58.39) circle (1.3pt);
\draw [fill=uuuuuu] (574.17,132.27) circle (1.3pt);
\draw [fill=uuuuuu] (553.25,142.53) circle (1.3pt);
\draw   (111.34,227.35) .. controls (111.34,212.96) and (123.39,201.3) .. (138.25,201.3) .. controls (153.11,201.3) and (165.15,212.96) .. (165.15,227.35) .. controls (165.15,241.74) and (153.11,253.4) .. (138.25,253.4) .. controls (123.39,253.4) and (111.34,241.74) .. (111.34,227.35) -- cycle ;
\draw   (65.72,299.55) .. controls (65.72,285.16) and (77.77,273.5) .. (92.63,273.5) .. controls (107.49,273.5) and (119.53,285.16) .. (119.53,299.55) .. controls (119.53,313.94) and (107.49,325.6) .. (92.63,325.6) .. controls (77.77,325.6) and (65.72,313.94) .. (65.72,299.55) -- cycle ;
\draw   (156.96,299.55) .. controls (156.96,285.16) and (169.01,273.5) .. (183.87,273.5) .. controls (198.73,273.5) and (210.78,285.16) .. (210.78,299.55) .. controls (210.78,313.94) and (198.73,325.6) .. (183.87,325.6) .. controls (169.01,325.6) and (156.96,313.94) .. (156.96,299.55) -- cycle ;
\draw  [pattern={Lines[angle=0,distance={2.3pt},line width=0.8pt]}, pattern color = uuuuuu, fill opacity=0.4]  (121.35,225.55) .. controls (134.67,224.99) and (143.23,220.33) .. (149.89,214.73) .. controls (146.09,249.24) and (160.36,278.16) .. (168.92,307.63) .. controls (150.84,277.22) and (143.23,253.9) .. (121.35,225.55) ;
\draw  [pattern={Lines[angle=90,distance={2.3pt},line width=0.8pt]}, pattern color = uuuuuu, fill opacity=0.4]  (197.46,294.95) .. controls (168.92,290.28) and (145.85,247.45) .. (121.35,225.55) .. controls (143.23,253.9) and (150.84,277.22) .. (168.92,307.63) .. controls (177.48,300.54) and (187.57,297.68) .. (197.46,294.95) ;
\draw [fill=uuuuuu] (197.46,294.95) circle (1.3pt);
\draw [fill=uuuuuu] (121.35,225.55) circle (1.3pt);
\draw [fill=uuuuuu] (168.92,307.63) circle (1.3pt);
\draw [fill=uuuuuu] (149.89,214.73) circle (1.3pt);
\draw   (304.46,228.28) .. controls (304.46,213.89) and (316.5,202.23) .. (331.36,202.23) .. controls (346.22,202.23) and (358.27,213.89) .. (358.27,228.28) .. controls (358.27,242.67) and (346.22,254.33) .. (331.36,254.33) .. controls (316.5,254.33) and (304.46,242.67) .. (304.46,228.28) -- cycle ;
\draw   (258.83,300.48) .. controls (258.83,286.1) and (270.88,274.43) .. (285.74,274.43) .. controls (300.6,274.43) and (312.65,286.1) .. (312.65,300.48) .. controls (312.65,314.87) and (300.6,326.53) .. (285.74,326.53) .. controls (270.88,326.53) and (258.83,314.87) .. (258.83,300.48) -- cycle ;
\draw   (350.08,300.48) .. controls (350.08,286.1) and (362.12,274.43) .. (376.98,274.43) .. controls (391.84,274.43) and (403.89,286.1) .. (403.89,300.48) .. controls (403.89,314.87) and (391.84,326.53) .. (376.98,326.53) .. controls (362.12,326.53) and (350.08,314.87) .. (350.08,300.48) -- cycle ;
\draw  [pattern={Lines[angle=0,distance={2.3pt},line width=0.8pt]}, pattern color = uuuuuu, fill opacity=0.4]  (317.32,231.52) .. controls (325.88,231.52) and (327.79,227.79) .. (332.54,223.12) .. controls (343.96,253.9) and (357.28,267.9) .. (376.98,300.48) .. controls (354.42,278.16) and (337.3,251.48) .. (317.32,231.52);
\draw [pattern={Lines[angle=90,distance={2.3pt},line width=0.8pt]}, pattern color = uuuuuu, fill opacity=0.4]   (332.54,223.12) .. controls (339.2,224.06) and (343.96,220.33) .. (346.81,214.73) .. controls (353.17,247.55) and (360.13,257.64).. (376.98,300.48) .. controls (357.28,267.9) and (343.96,253.9) .. (332.54,223.12)  ;
\draw [fill=uuuuuu] (332.54,223.12) circle (1.3pt);
\draw [fill=uuuuuu] (346.81,214.73) circle (1.3pt);
\draw [fill=uuuuuu] (376.98,300.48) circle (1.3pt);
\draw [fill=uuuuuu] (317.32,231.52) circle (1.3pt);
\draw   (493.77,230.15) .. controls (493.77,215.76) and (505.81,204.1) .. (520.67,204.1) .. controls (535.53,204.1) and (547.58,215.76) .. (547.58,230.15) .. controls (547.58,244.53) and (535.53,256.2) .. (520.67,256.2) .. controls (505.81,256.2) and (493.77,244.53) .. (493.77,230.15) -- cycle ;
\draw   (448.14,302.35) .. controls (448.14,287.96) and (460.19,276.3) .. (475.05,276.3) .. controls (489.91,276.3) and (501.96,287.96) .. (501.96,302.35) .. controls (501.96,316.74) and (489.91,328.4) .. (475.05,328.4) .. controls (460.19,328.4) and (448.14,316.74) .. (448.14,302.35) -- cycle ;
\draw   (539.39,302.35) .. controls (539.39,287.96) and (551.43,276.3) .. (566.29,276.3) .. controls (581.15,276.3) and (593.2,287.96) .. (593.2,302.35) .. controls (593.2,316.74) and (581.15,328.4) .. (566.29,328.4) .. controls (551.43,328.4) and (539.39,316.74) .. (539.39,302.35) -- cycle ;
\draw  [pattern={Lines[angle=0,distance={2.3pt},line width=0.8pt]}, pattern color = uuuuuu, fill opacity=0.4]  (513.29,242.71) .. controls (515.19,231.52) and (516.14,228.72) .. (517.1,218.46) .. controls (524.71,222.19) and (530.41,224.06) .. (538.02,227.79) .. controls (530.01,233.02) and (524.71,236.18) .. (513.29,242.71);
\draw  [pattern={Lines[angle=90,distance={2.3pt},line width=0.8pt]}, pattern color = uuuuuu, fill opacity=0.4]  (538.02,227.79) .. controls (542.78,258.57)  and (547.58,266.04).. (566.29,302.35) .. controls (545.63,275.92) and (537.07,257.64) .. (513.29,242.71) .. controls (524.71,236.18) and (530.01,233.02) .. (538.02,227.79) ;
\draw [fill=uuuuuu] (538.02,227.79) circle (1.3pt);
\draw [fill=uuuuuu] (566.29,302.35) circle (1.3pt);
\draw [fill=uuuuuu] (513.29,242.71) circle (1.3pt);
\draw [fill=uuuuuu] (517.1,218.46) circle (1.3pt);
\end{tikzpicture}
\caption{Six types of bad $\mathbb{S}_{2}$.}
\label{fig:BS2}
\end{figure}

The following result is an analogue of Proposition~\ref{PROP:B-vs-075M}, but instead of comparing bad edges and missing triples, we compare bad and missing $\mathbb{S}_{2}$'s.

\begin{proposition}\label{PROP:BS2-vs-09MS2}
    Suppose that $\mathcal{H}$ is an $n$-vertex $\{K_{4}^{3-}, C_{5}^{3-}\}$-free $3$-graph with $\rho(\mathbb{S}_{2},\mathcal{H}) \ge 6/13 - 10^{-6}$ and $V_{1} \cup V_{2} \cup V_{3} = V(\mathcal{H})$ is a locally maximal partition satisfying $\mu_{\mathcal{H}}(V_1, V_2, V_3) \ge 0.198$. 
    Then with $B\mathbb{S}_2 = B\mathbb{S}_2^{\mathcal{H}}(V_{1}, V_{2}, V_{3})$ and $M\mathbb{S}_2 = M\mathbb{S}_2^{\mathcal{H}}(V_{1}, V_{2}, V_{3})$ as defined in~\eqref{equ:DEF-BS2} and~\eqref{equ:DEF-MS2}, we have 
    \begin{align*}
        |B\mathbb{S}_2| - \frac{9}{10}|M\mathbb{S}_2| 
        = o(n^4).
    \end{align*}
\end{proposition}
 \begin{proof}
    Suppose to the contrary that this proposition fails. 
    Then there exists a constant $\delta > 0$ and an  increasing (i.e., the number of vertices is strictly increasing) infinite sequence  $\left(\mathcal{H}_{i}\right)_{i=1}^{\infty}$ of $\{K_{4}^{3-}, C_{5}^{3-}\}$-free $3$-graphs such that each $\mathcal{H}_{i}$ satisfies $\rho(\mathbb{S}_{2},\mathcal{H}_i) \ge 6/13 - 10^{-6}$ and admits a locally maximal partition $V_1^{(i)} \cup V_2^{(i)} \cup V_3^{(i)} = V(\mathcal{H}_{i})$ satisfying $\mu_{\mathcal{H}_i}(V_1^{(i)}, V_2^{(i)}, V_3^{(i)}) \ge 0.198$, but 
    \begin{align}\label{equ:BS2-09MS2-assume}
        |B\mathbb{S}_2^{(i)}| - \frac{9}{10}|M\mathbb{S}_2^{(i)}|
        \ge \delta n_i^{4},
    \end{align}
    where $n_i \coloneqq v(\mathcal{H}_{i})$, $B\mathbb{S}_2^{(i)} \coloneqq B\mathbb{S}_2^{\mathcal{H}}(V_1^{(i)}, V_2^{(i)}, V_3^{(i)})$, and $M\mathbb{S}_2^{(i)} \coloneqq M\mathbb{S}_2^{\mathcal{H}}(V_1^{(i)}, V_2^{(i)}, V_3^{(i)})$. 

    We would like to run flag algebra calculations on the limit $\phi$ of the sequence $\left( \mathcal{H}_{i} \right)_{i=1}^{\infty}$ in the theory of $\{K_{4}^{3-}, C_{5}^{3-}\}$-free $3$-graphs which are $3$-colored, that is, we have three unary relations $V_1, V_2, V_3$ with the only restriction that each vertex in a flag satisfies exactly one of them, equivalently, these unitary relations encode a $3$-partition of the vertex set.
     
    For each $i \in [3]$, let $(1,i)$ denote the $1$-vertex type where the color of the unique vertex is $i$.
    Consider the random homomorphism $\boldsymbol{\phi}^{(1,i)}$, which is the limit of taking a uniform random color-$i$ root. 
    Note that, by the max-cut ratio assumption $\mu_{\mathcal{H}_n}(V_1^{(n)}, V_2^{(n)}, V_3^{(n)}) \ge 0.198$, each part $V_i^{(n)}\subseteq V(\mathcal{H}_{n})$ is non-empty (in fact, it occupies a positive fraction of vertices), so $\boldsymbol{\phi}^{(1,i)}$ is well-defined.

    The local maximality assumptions translate, in the limit, to the statement that for each $i \in [3]$,
    \begin{align}\label{equ:flag-local-max-inequality}
        \boldsymbol{\phi}^{(1,i)}(K_{i+1,i+2}^{i})
        \ge \max\left\{ \boldsymbol{\phi}^{(1,i)}(K_{i,i+1}^{i}),~ \boldsymbol{\phi}^{(1,i)}(K_{i,i+2}^{i}) \right\}.
    \end{align}
    holds with probability $1$, where $K_{a,b}^{c}$ is the $(1,c)$-flag on three vertices that span an edge with the non-root vertices having colors $a$ and $b$ and the root vertex having color $c$. Here, the indices are taken modulo $3$.

    For every $n \ge 1$, since $\mu_{\mathcal{H}_n}(V_1^{(n)}, V_2^{(n)}, V_3^{(n)}) \ge 0.198$, it follows from Proposition~\ref{PROP:B-vs-075M} that 
    \begin{align*}
        B^{(n)} - \frac{3}{4}|M^{(n)}|
        \le 0, 
    \end{align*}
    where $B^{(n)} \coloneqq B_{\mathcal{H}_{n}}(V_1^{(n)}, V_2^{(n)}, V_3^{(n)})$ and $M^{(n)} \coloneqq M_{\mathcal{H}_{n}}(V_1^{(n)}, V_2^{(n)}, V_3^{(n)})$. 
    This translates, in the limit, to the statement that 
    \begin{align}\label{equ:flag-B-075M-inequality}
        \sum_{i \in [3]} \big( \phi(E_{i,i,i+1}) + \phi(E_{i,i,i+2}) \big) - \frac{3}{4}\phi(\xoverline{E}_{1,2,3})
        \le 0,
    \end{align}
    where $E_{a,b,c}$ denote the untyped flag on three vertices with colors $\{a,b,c\}$ containing exactly one edge, and $\xoverline{E}_{a,b,c}$ denote the untyped flag on three vertices with colors $\{a,b,c\}$ containing no edges. 


    Since each $\mathcal{H}_{n}$ is $\{K_{4}^{3-}, C_{5}^{3-}\}$-free, it follows from Theorem~\ref{THM:BLLP24-C5Minus-Turan-Density} that 
    \begin{align*}
        \rho(\mathcal{H}_{n}) 
        \le \frac{1}{4}
        \quad\text{and}\quad 
        \rho(\mathcal{H}_{n}[V_{i}^{(n)}])
        \le \frac{1}{4} \quad\text{for } i \in [3].
    \end{align*}
    This translates, in the limit, to the statement that 
    \begin{align}\label{equ:flag-Turan-density-inequality}
        \sum_{i,j,k}\phi(E_{i,j,k}) \le \frac{1}{4}
        \quad\text{and}\quad  
        3 \phi(E_{i,i,i}) - \phi(\xoverline{E}_{i,i,i})
        \le 0 \quad\text{for } i \in [3].
    \end{align} 
    %
    %
    The max-cut ratio $\mu_{\mathcal{H}_n}(V_1^{(n)}, V_2^{(n)}, V_3^{(n)}) \ge 0.198$ assumption translates, in the limit, to the statement that 
    \begin{align}\label{equ:flag-max-cut-inequality}
        \phi(E_{1,2,3}) 
        \ge 0.198.
    \end{align}
    Straightforward calculations show that the max-cut ratio $\mu_{\mathcal{H}_n}(V_1^{(n)}, V_2^{(n)}, V_3^{(n)}) \ge 0.198$ assumption implies that $|V_i^{(n)}| \in [v(\mathcal{H}_{n})/5,~v(\mathcal{H}_{n})/2]$ for $i \in [3]$. This translates, in the limit, to the statement that 
    \begin{align}\label{equ:flag-part-szie-inequality}
        \frac{1}{5}
        \le \phi(V_i) 
        \le \frac{1}{2} \quad\text{for } i \in [3], 
    \end{align}
    where $V_i$ denote the untyped color-$i$ flag with only one vertex. 
    
    The $\mathbb{S}_{2}$-density assumption $\rho(\mathbb{S}_{2},\mathcal{H}_n) \ge 6/13 - 10^{-6}$ assumption translates, in the limit, to the statement that 
    \begin{align}\label{equ:flag-S2-density-lower-inequality}
        \phi(\mathbb{S}) \ge \frac{6}{13} - 10^{-6},
    \end{align}
    where, for simplicity, $\mathbb{S}$ denotes the sum of all possible colored untyped $4$-vertex flags with exactly two edges.

    Finally,~\eqref{equ:BS2-09MS2-assume} translates, in the limit, to the statement that 
    \begin{align}\label{equ:flag-BS2-MS2-inequality}
        \phi(\mathbb{S}_{b}) - \frac{9}{10}\phi(\mathbb{S}_{m}) 
        \ge 4! \delta = 24 \delta,  
    \end{align}
    where, for simplicity, $\mathbb{S}_{b}$ denotes the sum of all possible colored untyped $4$-vertex flags with exactly two edges containing $E_{i,i,i+1}$ or $E_{i,i,i+2}$ for some $i \in [3]$ (see Figure~\ref{fig:BS2}), and $\mathbb{S}_{m}$ denotes the sum of all possible $3$-colored untyped $4$-vertex flags with at most one edge, which, if present, must be $E_{1,2,3}$ (see Figure~\ref{fig:MS2}).

    We can now run the usual flag algebra calculations where each of the inequalities in~\eqref{equ:flag-local-max-inequality} to~\eqref{equ:flag-S2-density-lower-inequality} can be multiplied by an unknown non-negative combination of respectively $0$-flags and $(1,i)$-flags (and then averaged out in the latter case). 
    The final inequality should prove that the left-hand side of~\eqref{equ:flag-BS2-MS2-inequality} is non-positive.

    Our proof uses $6$-vertex flags (where the number of linear constraints is $|\mathcal{F}_6^0|=16181$). 
    Running it on a conventional PC takes around $19$ hours.
    The results returned by the computer for the calculations for $\phi(\mathbb{S}_{b}) - \frac{9}{10}\phi(\mathbb{S}_{m})$ is indeed $0$, which contradicts~\eqref{equ:flag-BS2-MS2-inequality}. 
\end{proof}                                                                          
\section{Proof of Theorem~\ref{THM:C5Minus-S2}~\ref{THM:C5Minus-S2-a}}\label{SEC:proof-S2-density}    
We prove Theorem~\ref{THM:C5Minus-S2}~\ref{THM:C5Minus-S2-a} in this section.  Here (and in the subsequent sections), we will be rather loose with the constants in the lower-order terms.

Let us begin with the following crucial lemma.

\begin{lemma}\label{LEMMA:recursive-upper-bound}
    There exists a non-increasing function $N_{\ref{LEMMA:recursive-upper-bound}} \colon (0,1) \to \mathbb{N}$ such that  the following holds for every $\varepsilon > 0$ and for every $n \ge N_{\ref{LEMMA:recursive-upper-bound}}(\varepsilon)$.
    Suppose that $\mathcal{H}$ is a $\{K_{4}^{3-}, C_{5}^{3-}\}$-free $3$-graph on $n \ge N_{\ref{LEMMA:recursive-upper-bound}}$ vertices with $\rho(\mathbb{S}_{2}, \mathcal{H}) \ge 6/13 - 10^{-6}$. 
    Then there exists a partition $V_1 \cup V_2 \cup V_3 = V(\mathcal{H})$ with $|V_i| \in [n/5,~n/2]$ for $i \in [3]$ such that 
    \begin{align}\label{equ:LEMMA:recursive-upper-bound}
        \mathrm{N}(\mathbb{S}_{2}, \mathcal{H})
        \le \frac{|V_1||V_2||V_3| n}{2} + \sum_{i\in [3]}\mathrm{N}(\mathbb{S}_{2}, \mathcal{H}[V_i]) + \varepsilon n^4 - \max\left\{\frac{|B\mathbb{S}_{2}|}{9},~\frac{|M\mathbb{S}_{2}|}{10}\right\}, 
    \end{align}
    where $B\mathbb{S}_2 = B\mathbb{S}_2^{\mathcal{H}}(V_{1}, V_{2}, V_{3})$ and $M\mathbb{S}_2 = M\mathbb{S}_2^{\mathcal{H}}(V_{1}, V_{2}, V_{3})$ were defined in~\eqref{equ:DEF-BS2} and~\eqref{equ:DEF-MS2}.
\end{lemma}
\begin{proof}
    It is enough to show that, for each sufficiently large integer $m$, say $m\ge m_0$, there is $n_0(m)$ such that the conclusion holds when $\varepsilon =1/m$ and $n \ge n_0(m)$, as then we can take, for example, $N_{\ref{LEMMA:recursive-upper-bound}}(x) \coloneqq \max\{ n_0(m) \colon m_0\le m\le \lceil{1/x} \rceil\}$ for $x\in (0,1)$.
    
    Fix a sufficiently large $m$, let $\varepsilon \coloneqq 1/m$ and then let $n$ be sufficiently large. 
    Let $\mathcal{H}$ be an $n$-vertex $\{K_{4}^{3-}, C_{5}^{3-}\}$-free $3$-graph with 
    \begin{align}\label{equ:S2-density-lower-bound-assume}
        \rho(\mathbb{S}_{2}, \mathcal{H}) \ge \frac{6}{13} - 10^{-6}. 
    \end{align}
    Let $V \coloneqq V(\mathcal{H})$. 
    Let $V_1 \cup V_2 \cup V_3 = V$ be a partition such that $|\mathcal{H} \cap \mathcal{K}[V_1, V_2, V_3]|$ is maximized (in particular, the partition $V_1 \cup V_2 \cup V_3 = V$ is locally maximal). 
    Since $\rho(\mathbb{S}_{2}, \mathcal{H}) \ge 6/13 - 10^{-6}$, it follows from Proposition~\ref{PROP:max-cut-ratio-lower-bound} that 
    \begin{align}\label{equ:max-cut-ratio-lower-bound}
        \mu_{\mathcal{H}}(V_1, V_2, V_3)
        = \frac{6}{n^3} \cdot |\mathcal{H} \cap \mathcal{K}[V_1, V_2, V_3]|
        \ge \alpha_{\ref{pr2}}.
    \end{align}
    Let $x_i \coloneqq |V_i|/n$ for $i \in [3]$. 
    It follows from~\eqref{equ:max-cut-ratio-lower-bound} that $x_1 x_2 x_3 \ge \alpha_{\ref{pr2}}/6$. 
    Straightforward calculations show that 
    \begin{align}\label{equ:size-Vi-lower-upper-bounds}
        \frac{1}{5} < x_i < \frac{1}{2} \quad\text{for every $i \in [3]$}. 
    \end{align}
    This shows that $|V_i| \in [n/5,~n/2]$ for $i \in [3]$. 
    
    Since $\mathcal{H}$ satisfies~\eqref{equ:S2-density-lower-bound-assume} and the partition $V = V_1 \cup V_2 \cup V_3$ satisfies~\eqref{equ:max-cut-ratio-lower-bound}, it follows from Proposition~\ref{PROP:BS2-vs-09MS2} that 
    \begin{align*}
        |B\mathbb{S}_2| - \frac{9}{10}|M\mathbb{S}_2| 
        \le \frac{\varepsilon n^4}{2}. 
    \end{align*}
    It follows that 
    \begin{align*}
        |B\mathbb{S}_2| - |M\mathbb{S}_2|
        & = |B\mathbb{S}_2| - \frac{9}{10}|M\mathbb{S}_2| - \frac{|M\mathbb{S}_2|}{10} 
        \le \frac{\varepsilon n^4}{2} - \frac{|M\mathbb{S}_2|}{10} \quad\text{and} \\[0.5em]
        |B\mathbb{S}_2| - |M\mathbb{S}_2|
        & = \frac{10}{9}\left( |B\mathbb{S}_2| - \frac{9}{10}|M\mathbb{S}_2| \right) - \frac{|B\mathbb{S}_2|}{9} 
        \le \varepsilon n^4 - \frac{|B\mathbb{S}_2|}{9},
    \end{align*}
    which implies that 
    \begin{align}\label{equ:BS2-vs-09MS2}
        |B\mathbb{S}_2| - |M\mathbb{S}_2|
        \le \varepsilon n^4 - \max\left\{|M\mathbb{S}_2|/10,~|B\mathbb{S}_2|/9\right\}. 
    \end{align}
    Also, note that 
    \begin{align*}
        \mathrm{N}(\mathbb{S}_{2}, \mathcal{K}[V_1, V_2, V_3])
        & = |V_1||V_2|\binom{|V_3|}{2} + |V_1|\binom{|V_2|}{2}|V_3| + \binom{|V_1|}{2}|V_2||V_3| \\
        & \le \frac{1}{2}\left(|V_1||V_2||V_3|^2 + |V_1||V_2|^2|V_3| + |V_1|^2|V_2||V_3|\right) \\
        & = \frac{1}{2} \left(|V_1||V_2||V_3| \left(|V_1|+ |V_2| + |V_3| \right)\right)
        = \frac{|V_1||V_2||V_3|n}{2}. 
    \end{align*}
    Combining it with~\eqref{equ:BS2-vs-09MS2}, we obtain 
    \begin{align*}
        \mathrm{N}(\mathbb{S}_{2}, \mathcal{H})
        & = \mathrm{N}(\mathbb{S}_{2}, \mathcal{H} \cap \mathcal{K}[V_1, V_2, V_3]) + |B\mathbb{S}_2| + \sum_{i\in[3]}\mathrm{N}(\mathbb{S}_{2}, \mathcal{H}[V_i]) \\
        & = \mathrm{N}(\mathbb{S}_{2}, \mathcal{K}[V_1, V_2, V_3]) - |M\mathbb{S}_2| + |B\mathbb{S}_2| + \sum_{i\in[3]}\mathrm{N}(\mathbb{S}_{2}, \mathcal{H}[V_i]) \\
        & \le \frac{|V_1||V_2||V_3|n}{2} + \sum_{i\in[3]}\mathrm{N}(\mathbb{S}_{2}, \mathcal{H}[V_i]) + \varepsilon n^4 - \max\left\{\frac{|B\mathbb{S}_{2}|}{9},~\frac{|M\mathbb{S}_{2}|}{10}\right\}, 
    \end{align*}
    which proves Lemma~\ref{LEMMA:recursive-upper-bound}.
\end{proof}

We are now ready to present the proof of Theorem~\ref{THM:C5Minus-S2}~\ref{THM:C5Minus-S2-a}. 
\begin{proof}[Proof of Theorem~\ref{THM:C5Minus-S2}~\ref{THM:C5Minus-S2-a}]
    Let $\alpha\coloneqq
    \pi(\mathbb{S}_{2}, \{K_{4}^{3-}, C_5^{3-}\})$. Fix a small constant $\varepsilon > 0$ and let $n$ be sufficiently large. Let $\mathcal{H}$ be an $n$-vertex $\{K_{4}^{3-}, C_5^{3-}\}$-free $3$-graph with $\mathbb{N}(\mathbb{S}_{2}, \mathcal{H}) = \mathrm{ex}(n,\mathbb{S}_{2}, \{K_{4}^{3-}, C_5^{3-}\})$, that is, the maximum number of $\mathbb{S}_{2}$. 
    Let $V \coloneqq V(\mathcal{H})$. 

    It follows from the recursive $3$-partition construction $T_{\mathrm{rec}}$ that $\alpha \ge 6/13$. 
    Therefore, by taking $n$ sufficiently large, we can ensure that
    \begin{align}\label{equ:S2-number-lower-bound-extremal}
        \mathbb{N}(\mathbb{S}_{2}, \mathcal{H}) 
        = \mathrm{ex}(n,\mathbb{S}_{2}, \{K_{4}^{3-}, C_5^{3-}\})
        \ge (\alpha - \varepsilon) \frac{n^4}{24}
        \ge \left(\frac{6}{13} - 10^{-6}\right) \frac{n^4}{24}. 
    \end{align}
    Then it follows from Lemma~\ref{LEMMA:recursive-upper-bound} that there exists a partition $V_1 \cup V_2 \cup V_3 = V$ with $|V_i| \in [n/5,~n/2]$ for $i \in [3]$ such that~\eqref{equ:LEMMA:recursive-upper-bound} holds. 
    
    Let $x_i \coloneqq |V_i|/n$ for $i \in [3]$, noting that $x_i \in [1/5,~1/2]$. 
    For each $i \in [3]$, since $|V_i| \ge n/5$, we can choose $n$ sufficiently large so that 
    \begin{align*}
        \mathrm{N}(\mathbb{S}_{2}, \mathcal{H}[V_i])
        \le \left(\alpha + \varepsilon \right) \frac{|V_i|^4}{24}
        = \frac{(\alpha + \varepsilon) x_i^4 n^4}{24}
        \le \frac{\alpha x_i^4 n^4}{24} + \frac{\varepsilon n^4}{24}.
    \end{align*}
    Combining it with~\eqref{equ:S2-number-lower-bound-extremal} and~\eqref{equ:LEMMA:recursive-upper-bound}, we obtain 
    \begin{align*}
        (\alpha - \varepsilon) \frac{n^4}{24}
        & \le \frac{x_1 x_2 x_3 n^4}{2} + \sum_{i\in [3]} \frac{\alpha x_i^4 n^4 + \varepsilon n^4}{24} + \varepsilon n^4 \\
        & \le \left(\frac{x_1 x_2 x_3}{2} + \frac{\alpha (x_1^4 + x_2^4 + x_3^4)}{24}\right) n^4 + \frac{9 \varepsilon n^4}{8}, 
    \end{align*}
    Simplifying this inequality, we obtain 
    \begin{align*}
        \alpha 
        \le 12 x_1 x_2 x_3 + \alpha (x_1^4 + x_2^4 + x_3^4) + 28 \varepsilon,
    \end{align*}
    which, by Fact~\ref{FACT:inequalities}, implies that 
    \begin{align*}
        \alpha 
        & \le \frac{12 x_1 x_2 x_3}{1 - (x_1^4 + x_2^4 + x_3^4)} + \frac{28 \varepsilon}{1- (x_1^4 + x_2^4 + x_3^4)} \\
        & \le \frac{12}{26} + \frac{28 \varepsilon}{1- \left((1/2)^4 + (1/2)^4 + (1/2)^4 \right)}
        \le \frac{6}{13} + 35 \varepsilon.
    \end{align*}
    Letting $\varepsilon \to 0$, we obtain $\alpha \le 6/13$, which proves that $\pi(\mathbb{S}_{2}, \{K_{4}^{3-}, C_5^{3-}\}) = 6/13$. 
\end{proof}

\section{Proof of Theorem~\ref{THM:C5Minus-S2}~\ref{THM:C5Minus-S2-b}}\label{SEC:proof-S2-stability}    
In this section, we prove Theorem~\ref{THM:C5Minus-S2}~\ref{THM:C5Minus-S2-b}.
We begin by establishing the following weaker form of stability. 

\begin{lemma}\label{LEMMA:weak-stability}
    There exist an absolute constant $\varepsilon_{\ref{LEMMA:weak-stability}} > 0$ and a non-increasing function $N_{\ref{LEMMA:weak-stability}} \colon (0,\varepsilon_{\ref{LEMMA:weak-stability}}) \to \mathbb{N}$ such that  the following holds for every $\varepsilon \in (0, \varepsilon_{\ref{LEMMA:weak-stability}})$ and for every $n \ge N_{\ref{LEMMA:weak-stability}}(\varepsilon)$.
    Suppose that $\mathcal{H}$ is an $n$-vertex $\{K_{4}^{3-}, C_{5}^{3-}\}$-free $3$-graph with $\mathrm{N}(\mathbb{S}_{2},\mathcal{H}) \ge \left(\pi(\mathbb{S}_{2}, \{K_{4}^{3-}, C_{5}^{3-}\}) - \varepsilon \right)n^4/24$. 
    Then there exists a partition $V_{1} \cup V_{2} \cup V_{3} =V(\mathcal{H})$ such that
      \begin{enumerate}[label=(\roman*)]
          \item\label{LEMMA:weak-stability-1} $\left|| V_{i}|-n/3 \right| \le 6\varepsilon^{1/2} n$ for every $i \in[3]$, 
          \item\label{LEMMA:weak-stability-2}  $\max\{|B|,~|M|\} \le 25 \varepsilon n^{3}$, where $B=B_{\mathcal{H}}(V_{1}, V_{2}, V_{3})$ and $M=M_{\mathcal{H}}(V_{1}, V_{2}, V_{3})$, and 
          \item\label{LEMMA:weak-stability-3} $\mathrm{N}(\mathbb{S}_{2},\mathcal{H}[V_i]) \geq \left(\pi(\mathbb{S}_{2}, \{K_{4}^{3-}, C_{5}^{3-}\}) - 2400 \varepsilon \right)\left|V_{i}\right|^{4}/24$ for every $i \in[3]$.
      \end{enumerate}
\end{lemma}
\begin{proof}
    Let $\varepsilon_{\ref{LEMMA:weak-stability}} > 0$ be sufficiently small. 
    Fix $\varepsilon \in (0,\varepsilon_{\ref{LEMMA:weak-stability}})$ and let $N_{\ref{LEMMA:weak-stability}}(\varepsilon)$ be sufficiently large. The monotonicity of the function $N_{\ref{LEMMA:weak-stability}} \colon (0,\varepsilon_{\ref{LEMMA:weak-stability}}) \to \mathbb{N}$ can be ensured using the same trick as in the proof of Lemma~\ref{LEMMA:recursive-upper-bound}. 
    
    Let $\alpha \coloneqq \pi(\mathbb{S}_{2}, \{K_{4}^{3-}, C_{5}^{3-}\}) = 6/13$. 
    Let $\mathcal{H}$ be a $\{K_{4}^{3-}, C_{5}^{3-}\}$-free $3$-graph on $n \ge N_{\ref{LEMMA:weak-stability}}(\varepsilon)$ vertices with 
    \begin{align}\label{equ:S2-lower-bound-assume-b}
        \mathrm{N}(\mathbb{S}_{2},\mathcal{H}) 
        \ge (\alpha - \varepsilon){n^4}/{24}. 
    \end{align}
    By Lemma~\ref{LEMMA:recursive-upper-bound}, there exists a partition  $V_1 \cup V_2 \cup V_3 = V \coloneqq V(\mathcal{H})$ with $|V_i| \in [n/5,~n/2]$ for every $i \in [3]$ such that~\eqref{equ:LEMMA:recursive-upper-bound} holds. 

    Let $x_i \coloneqq |V_i|/n$ for $i \in [3]$. 
    For each $i \in [3]$, since $|V_i| \ge n/5$, we can choose $n$ sufficiently large so that $\mathrm{N}(\mathbb{S}_{2}, \mathcal{H}[V_i]) \le \left(\alpha + \varepsilon \right) |V_i|^4/24 \le \alpha |V_i|^4/24 + \varepsilon n^4/24$.
    For each $i \in [3]$, let $y_i \ge 0$ denote the real number such that 
    \begin{align*}
        \mathrm{N}(\mathbb{S}_{2}, \mathcal{H}[V_i])
        = \frac{\alpha |V_i|^4}{24} + \frac{\varepsilon n^4}{24} - y_i n^4
        = \left(\frac{\alpha x_i^4}{24} + \frac{\varepsilon}{24} - y_i \right) n^4.
    \end{align*}
    Let $z \ge 0$ denote the real number such that 
    \begin{align*}
        z n^4
        = \max\left\{{|B\mathbb{S}_{2}|}/{9},~{|M\mathbb{S}_{2}|}/{10}\right\}. 
    \end{align*}
    Combining~\eqref{equ:S2-lower-bound-assume-b} and~\eqref{equ:LEMMA:recursive-upper-bound}, we obtain 
    \begin{align*}
        (\alpha - \varepsilon) \frac{n^4}{24}
        & \le \frac{x_1 x_2 x_3 n^4}{2} + \sum_{i\in [3]} \left(\frac{\alpha x_i^4}{24} + \frac{\varepsilon}{24} - y_i \right) n^4 + \varepsilon n^4 - z n^4\\
        & \le \left(\frac{x_1 x_2 x_3}{2} + \frac{\alpha (x_1^4 + x_2^4 + x_3^4)}{24}\right) n^4 + \frac{9 \varepsilon n^4}{8} - \left(y_1+y_2+y_3 + z\right) n^4. 
    \end{align*}
    Simplifying this inequality and applying Fact~\ref{FACT:inequalities}~\ref{FACT:inequalities-3}, we obtain 
    \begin{align*}
        \frac{\alpha}{24}
        & \le \frac{x_1 x_2 x_3}{2} + \frac{\alpha (x_1^4 + x_2^4 + x_3^4)}{24} + \frac{7 \varepsilon}{6} - \left(y_1+y_2+y_3 + z\right) \\
        & \le \frac{\alpha}{24} - \frac{1}{30}\sum_{i\in [3]}\left(x_i - \frac{1}{3}\right)^2 + \frac{7 \varepsilon}{6} - \left(y_1+y_2+y_3 + z\right).
    \end{align*}
    It follows that 
    \begin{align*}
        \frac{1}{30} \sum_{i\in [3]}\left(x_i - \frac{1}{3}\right)^2 
        \le  \frac{7 \varepsilon}{6} 
        \quad\text{and}\quad 
        \max\{y_1,~y_2,~y_3,~z\} 
        \le \frac{7 \varepsilon}{6}.
    \end{align*}
    The first inequality implies that $|x_i - 1/3|\le \sqrt{35 \varepsilon} \le 6 \varepsilon^{1/2}$ for every $i \in [3]$, which proves~\ref{LEMMA:weak-stability-1}.
    The inequality $\max\{y_1,~y_2,~y_3\} \le 7\varepsilon/6$ implies that for every $i \in [3]$,
    \begin{align*}
        \mathrm{N}(\mathbb{S}_{2}, \mathcal{H}[V_i])
        & = \frac{\alpha |V_i|^4}{24} + \frac{\varepsilon n^4}{24} - y_i n^4 \\
        & \ge \frac{\alpha |V_i|^4}{24} - \frac{9 \varepsilon}{8}n^4
        \ge \left(\frac{\alpha}{24} - \frac{9 \varepsilon}{8 (1/3-\varepsilon)^4}\right) |V_i|^4
        \ge \left(\frac{\alpha}{24} - 100 \varepsilon \right) |V_i|^4,
    \end{align*}
    which proves~\ref{LEMMA:weak-stability-3}.
    Here, we used the conclusion that $|V_i| \ge (1/3 - \varepsilon) n$ and $\varepsilon$ is sufficiently small.

    Now we consider the inequality $z \le 7 \varepsilon/6$. 
    Let $\beta \coloneqq |M|/n^3$. 
    Note that for every missing triple $e\in M$, there are exactly $|V_1| - 1 + |V_2| - 1 + |V_3| - 1 = n-3$ missing $\mathbb{S}_{2}$ (in $M\mathbb{S}_{2}$) containing $e$, while every missing $\mathbb{S}_{2}$ in $M\mathbb{S}_{2}$ contains at most two missing triples in $M$.
    It follows that 
    \begin{align*}
        \beta n^3 \cdot (n-3)
        = |M| \cdot (n-3)
        \le 2 \cdot |M\mathbb{S}_{2}|
        \le 2 \cdot 10 z n^4,
    \end{align*}
    which implies that 
    \begin{align*}
        \beta 
        \le \frac{20 z n^4}{n^3 (n-3)}
        \le 21 z
        \le 21 \cdot \frac{7 \varepsilon}{6}
        \le 25 \varepsilon. 
    \end{align*}
    This proves that $|M| \le 25 \varepsilon n^3$.
    
    Recall from the proof of Lemma~\ref{LEMMA:recursive-upper-bound} that the partition $V_1 \cup V_2 \cup V_3 = V(\mathcal{H})$ is maximum.  
    It follows from~\eqref{equ:S2-lower-bound-assume-b} and Proposition~\ref{PROP:max-cut-ratio-lower-bound} that $\mu_{\mathcal{H}}(V_1, V_2, V_3) \ge 0.918$. 
    Combining this with Proposition~\ref{PROP:B-vs-075M}, we obtain that $|B| \le 3|M|/4 \le 25 \varepsilon n^3$.
    This completes the proof of~\ref{LEMMA:weak-stability-2}. 
\end{proof}

We are now ready to present the proof of Theorem~\ref{THM:C5Minus-S2}~\ref{THM:C5Minus-S2-b}. 
\begin{proof}[Proof of Theorem~\ref{THM:C5Minus-S2}~\ref{THM:C5Minus-S2-b}]
    Let $\alpha \coloneqq \pi(\mathbb{S}_{2}, \{K_{4}^{3-}, C_5^{3-} \}) = 6/13$. 
    Fix $\varepsilon>0$. We may assume that $\varepsilon$ is sufficiently small. Let $\delta\coloneqq\varepsilon^{13}/1200$. Let $n$ be sufficiently large; in particular, we can assume that
    \begin{align*}
        \varepsilon n\ge \max\{N_{\ref{LEMMA:recursive-upper-bound}}(\delta),~N_{\ref{LEMMA:weak-stability}}(\delta)\},
    \end{align*}
    where $N_{\ref{LEMMA:recursive-upper-bound}}\colon (0,1) \to \mathbb{N}$ and $N_{\ref{LEMMA:weak-stability}}\colon (0, \varepsilon_{\ref{LEMMA:weak-stability}}) \to \mathbb{N}$ are the functions returned by Lemmas~\ref{LEMMA:recursive-upper-bound} and~\ref{LEMMA:weak-stability}, respectively. 

    We will prove by induction on $m$ that for all $m \le n$, every $m$-vertex $\{K_{4}^{3-}, C_5^{3-} \}$-free $3$-graph $\mathcal{H}$ satisfying
    \begin{align}\label{equ:induction-assumption}
        \mathrm{N}(\mathbb{S}_{2}, \mathcal{H})
        \ge\left(\alpha - \delta \left(\frac{n}{m}\right)^{12}\right)\frac{m^4}{24} 
    \end{align}
    can be transformed into a $T_{\mathrm{rec}}$-subconstruction by removing an edge set of size at most 
    \begin{align*}
        \frac{600 \delta m^{3}}{\varepsilon^{12}} +\frac{\varepsilon n^{2} m}{6}.
    \end{align*}
    The base case $m<\varepsilon n$ holds trivially, since 
    \begin{align*}
        \binom{m}{3} 
        \le \binom{\varepsilon n}{3} 
        \le \frac{\varepsilon^3 n^3}{6} 
        \le \frac{\varepsilon n^{2} m}{6}.
    \end{align*}
    So we may assume that $m \ge\varepsilon n$. 
    
    Let $\xi \coloneqq \delta(n/m)^{12}$, noting that 
    \begin{align*}
        \xi 
        = \delta\left(\frac{n}{m}\right)
        \geq \delta
        = \frac{\varepsilon^{13}}{1200} 
        \quad\text{and}\quad
        \xi 
        = \delta\left(\frac{n}{m}\right)
        \leq \delta\left(\frac{n}{\varepsilon n}\right)^{12} 
        \leq \frac{\delta}{\varepsilon^{12}}=\frac{\varepsilon}{1200} 
        \ll 1.
    \end{align*}
    Additionally, note that 
    \begin{align*}
        m 
        \geq \varepsilon n 
        \gg \max\{N_{\ref{LEMMA:recursive-upper-bound}}(\delta),~N_{\ref{LEMMA:weak-stability}}(\delta)\} 
        \geq \max\{N_{\ref{LEMMA:recursive-upper-bound}}(\xi),~N_{\ref{LEMMA:weak-stability}}(\xi)\}.
    \end{align*}
    Let $\mathcal{H}$ be an arbitrary $m $-vertex $\{K_{4}^{3-}, C_5^{3-} \}$-free $3$-graph satisfying~\eqref{equ:induction-assumption}, that is, $\mathrm{N}(\mathbb{S}_{2}, \mathcal{H}) \ge (\alpha - \xi) m^4/24$. 
    Applying Lemma~\ref{LEMMA:weak-stability} to $\mathcal{H}$, we obtain a partition $V(\mathcal{H})=V_{1} \cup V_{2} \cup V_{3}$ such that 
     \begin{enumerate}[label=(\roman*)]
          \item\label{item:stability-1} $|V_i| \in [m/5,~m/2]$ for every $i \in[3]$, 
          \item\label{item:stability-2} $|B|\le 25 \xi m^{3}$, where $B=B_{\mathcal{H}}(V_1,V_2,V_3)$,  and
          \item\label{item:stability-3} $\mathrm{N}(\mathbb{S}_{2}, \mathcal{H}[V_i]) \geq (\alpha - 2400 \xi)\left|V_{i}\right|^{4}/24$ for every $i \in[3]$.
      \end{enumerate}
    For every $i \in[3]$, since $|V_{i}| \le m/2$, it follows from~\ref{item:stability-3} and the definition of $\xi$ that 
    \begin{align*}
        \mathrm{N}(\mathbb{S}_{2}, \mathcal{H}[V_i]) 
        & \geq \left( \alpha - 2400 \xi \right){|V_i|^4}/{24} \\
        & = \left( \alpha - 2400 \delta\left(\frac{n}{m}\right)^{12} \right)\frac{|V_i|^4}{24} \\
        & \ge \left( \alpha - 2400 \delta\left(\frac{n}{2|V_i|}\right)^{12} \right)\frac{|V_i|^4}{24} 
        \ge \left( \alpha - \delta\left(\frac{n}{|V_i|}\right)^{12} \right)\frac{|V_i|^4}{24}. 
    \end{align*}
    It follows from the inductive hypothesis that each $\mathcal{H}[V_i]$ can be transformed into a $T_{\mathrm{rec}}$-subconstruction after removing an edge set $\mathcal{E}_{i}$ of size at most 
    \begin{align*}
        \frac{600 \delta |V_{i}|^{3}}{\varepsilon^{12}} + \frac{\varepsilon n^{2}|V_{i}|}{6}
    \end{align*}
    Combining this with~\ref{item:stability-2}, we conclude that $\mathcal{H}$ can be transformed into a $T_{rec }$-subconstruction after removing at most 
    \begin{align*}
        |B| + \sum_{i \in[3]}\left( \frac{600 \delta |V_{i}|^{3}}{\varepsilon^{12}} + \frac{\varepsilon n^{2}|V_{i}|}{6} \right)
        & \leq 25 \delta\left(\frac{n}{m}\right)^{12} m^{3}+3 \cdot \frac{600 \delta}{\varepsilon^{12}}\left(\frac{m}{2}\right)^{3}+\frac{\varepsilon n^{2} m}{6} \\
        & \leq 25 \delta\left(\frac{1}{\varepsilon}\right)^{12} m^{3}+\frac{3}{8} \cdot \frac{600 \delta}{\varepsilon^{12}} m^{3}+\frac{\varepsilon n^{2} m}{6} \\
        & = \frac{250 \delta}{\varepsilon^{12}} m^{3}+\frac{\varepsilon n^{2} m}{6} 
        \le \frac{600 \delta}{\varepsilon^{12}} m^{3}+\frac{\varepsilon n^{2} m}{6}
    \end{align*}
    edges. This completes the proof for the inductive step.

    Taking $m=n$ in~\eqref{equ:induction-assumption}, we conclude that every $n$-vertex $\{K_{4}^{3-}, C_5^{3-} \}$-free $3$-graph with $\mathrm{N}(\mathbb{S}_{2}, \mathcal{H}) \ge (\alpha - \delta)n^4/24$ is a $T_{\mathrm{rec}}$-subconstruction after removing at most 
    \begin{align*}
        \frac{600 \delta n^3}{\varepsilon^{12}} +\frac{\varepsilon n^{3}}{6} \leq \frac{\varepsilon n^{3}}{2}+\frac{\varepsilon n^{3}}{6}<\varepsilon n^{3}
    \end{align*}
    edges. This completes the proof of Theorem~\ref{THM:C5Minus-S2}.
\end{proof}

\section{Concluding remarks}\label{SEC:Remark}
Theorem~\ref{THM:Main-result-C5Minus-L2Norm} motivates the following natural extension, for which the method used in~\cite{BL24tightcycle} may be useful in the case of large $\ell$. 
\begin{problem}\label{PROB:C5Minus-Lp-Norm}
    Let $p \ge 3$ be an integer. 
    Determine $\pi_{\ell_p}(C_{\ell}^{3-})$ for all integers $\ell \ge 5$ satisfying $\ell \not\equiv 0 \pmod 3$.
\end{problem}

Let $F_{3,2}$ denote the $3$-graph on vertex set $[5]$ with edges 
\begin{align*}
    \big\{ \{1,2,3\},~\{1,2,4\},~\{1,2,5\},~\{3,4,5\} \big\}.
\end{align*}
Confirming a conjecture of Mubayi--R\"{o}dl~\cite{MR02}, F{\"u}redi--Pikhurko--Simonovits~\cite{FPS03,FPS05} determined the Tur\'{a}n density and the exact Tur\'{a}n number (for large $n$) of $F_{3,2}$.
In~\cite{BCL22a}, Balogh--Clemen--Lidick\'{y}, using computer-assisted flag algebra calculations, determined the $\ell_{2}$-norm Tur\'{a}n density of $F_{3,2}$ up to an additive error of order $10^{-9}$. 
Using the newly developed package, we are now able to determine the exact value of $\pi_{\ell_{2}}(F_{3,2})$. 
The script and certificate can be found at \href{https://drive.google.com/drive/folders/17CVMBJ60S81CSYK54fq4jyb0K3oa3Pnp?usp=share_link}{\url{https://drive.google.com/drive/folders/17CVMBJ60S81CSYK54fq4jyb0K3oa3Pnp?usp=share_link}}.

\begin{theorem}\label{THM:F32-L2-Norm}
    We have $\pi_{\ell_{2}}(F_{3,2}) = 1/8$. 
\end{theorem}

Let $V_1$ and $V_2$ be two disjoint sets, and let $\mathbb{B}(V_1, V_2)$ denote the $3$-graph with vertex set $V_1 \cup V_2$ consisting of all triples that contain exactly two vertices from $V_1$ (and hence exactly one vertex from $V_2$).
The lower bound in Theorem~\ref{THM:F32-L2-Norm} is achieved by the construction $\mathbb{B}(V_1, V_2)$ with the parts satisfy the ratio $|V_1| : |V_2| = \sqrt{2}+1 +o(1)$. 

Further analysis of the output from the flag algebra computations suggests that $F_{3,2}$ also exhibits the Erd\H{o}s--Simonovits-type stability property in the $\ell_2$-norm. 
Unfortunately, $F_{3,2}$ does not possess the vertex-extendability property introduced in~\cite{LMR23b,CL24} (see~{\cite[Construction~1.2]{FPS05}} for a counterexample), and thus the framework developed in~\cite{CL24} cannot be directly applied to determine the exact value of $\mathrm{ex}_{\ell_2}(n,F_{3,2})$ for large $n$. 
This makes the determination of $\mathrm{ex}_{\ell_2}(n,F_{3,2})$ somewhat more involved, and we hope to return to this topic in future work. 

\bibliographystyle{alpha}
\bibliography{C5minus}
\end{document}